\documentclass[11pt,leqno]{amsart}
\usepackage{amssymb,amsmath,amsthm,amsfonts}
\usepackage{latexsym}
\usepackage{color}
\allowdisplaybreaks

\newcommand{\Hmm}[1]{\leavevmode{\marginpar{\tiny%
$\hbox to 0mm{\hspace*{-0.5mm}$\leftarrow$\hss}%
\vcenter{\vrule depth 0.1mm height 0.1mm width \the\marginparwidth}%
\hbox to 0mm{\hss$\rightarrow$\hspace*{-0.5mm}}$\\\relax\raggedright #1}}}
\newcommand{\nc}{\newcommand}
\nc{\les}{\lesssim}
\nc{\nit}{\noindent}
\nc{\nn}{\nonumber}
\nc{\D}{\partial}
\nc{\diff}[2]{\frac{d #1}{d #2}}
\nc{\diffn}[3]{\frac{d^{#3} #1}{d {#2}^{#3}}}
\nc{\pdiff}[2]{\frac{\partial #1}{\partial #2}}
\nc{\pdiffn}[3]{\frac{\partial^{#3} #1}{\partial{#2}^{#3}}}
\nc{\abs}[1] {\lvert #1 \rvert}
\nc{\cAc}{{\cal A}_c}
\nc{\cE}{{\cal E}}
\nc{\cF}{{\mathcal F}}
\nc{\cP}{{\cal P}}
\nc{\cV}{{\cal V}}
\nc{\cQ}{{\cal Q}}
\nc{\cGin}{{\cal G}_{\rm in}}
\nc{\cGout}{{\cal G}_{\rm out}}
\nc{\cO}{{\cal O}}
\nc{\Lav}{{\cal L}_{\rm av}}
\nc{\cL}{{\cal L}}
\nc{\cB}{{\cal B}}
\nc{\cZ}{{\cal Z}}
\nc{\cR}{{\cal R}}
\nc{\cT}{{\cal T}}
\nc{\cY}{{\cal Y}}
\nc{\cX}{{\cal X}}
\nc{\cXT}{{{\cal X}(T)}}
\nc{\cBT}{{{\cal B}(T)}}
\nc{\vD}{{\vec \mathcal{D}}}
\nc{\efield}{\mathcal{E}}
\nc{\vE}{{\vec \efield}}
\nc{\vB}{{\vec \mathcal{B}}}
\nc{\vH}{{\vec \mathcal{H}}}
\nc{\ty}{{\tilde y}}
\nc{\tu}{{\tilde u}}
\nc{\tV}{{\tilde V}}
\nc{\Pc}{{\bf P_c}}
\nc{\bx}{{\bf x}}
\nc{\bX}{{\bf X}}
\nc{\bXYZ}{{\bf XYZ}}
\nc{\bY}{{\bf Y}}
\nc{\bF}{{\bf F}}
\nc{\bS}{{\bf S}}
\nc{\dV}{{\delta V}}
\nc{\dE}{{\delta E}}
\nc{\TT}{{\Theta}}
\nc{\dPsi}{{\delta\Psi}}
\nc{\order}{{\cal O}}
\nc{\Rout}{R_{\rm out}}
\nc{\eplus}{e_+}
\nc{\eminus}{e_-}
\nc{\epm}{e_\pm}
\nc{\eps}{\varepsilon}
\nc{\vnabla}{{\vec\nabla}}
\nc{\G}{\Gamma}
\nc{\w}{\omega}
\nc{\mh}{h}
\nc{\mg}{g}
\nc{\vphi}{\varphi}
\nc{\tlambda}{\tilde\lambda}
\nc{\be}{\begin{equation}}
\nc{\ee}{\end{equation}}
\nc{\ba}{\begin{eqnarray}}
\nc{\ea}{\end{eqnarray}}

\nc{\g}{\gamma}
\nc{\ol}{\overline}
\newcommand{\n}{\nu}

\newtheorem{theorem}{Theorem}[section]
\newtheorem{lemma}[theorem]{Lemma}
\newtheorem{prop}[theorem]{Proposition}
\newtheorem{corollary}[theorem]{Corollary}
\newtheorem{defin}[theorem]{Definition}
\newtheorem{rmk}[theorem]{Remark}

\def\R{\mathbb R}

\nc{\T}{\mathbb T}
\nc{\Z}{\mathbb Z}
\nc{\N}{\mathbb N}
\nc{\pt}{\partial_t}
\nc{\la}{\langle}
\nc{\ra}{\rangle}
\nc{\infint}{\int_{-\infty}^{\infty}}
\nc{\halfwidth}{6.5cm}
\nc{\figwidth}{10cm}

\nc{\nlayers}{L} \nc{\nsectors}{M}
\nc{\indicator}{\mathbf{1}}
\nc{\Rhole}{R_{\rm hole}}
\nc{\Rring}{R_{\rm ring}}
\nc{\neff}{n_{\rm eff}}
\nc{\Frem}{F_{\rm rem}}
\nc{\DD}{\Delta}
\nc{\cD}{\mathcal D}
\nc{\lnorm}{\left\|}
\nc{\rnorm}{\right\|}
\nc{\rnormp}{\right\|_{\ell^{p,\eps}}}
\nc{\rar}{\rightarrow}
\nc{\sgn}{{\rm sign}}
\allowdisplaybreaks
\date{\today}

\begin{document}

\title[Dynamics of the Zakharov System]{Smoothing and Global Attractors for the Zakharov System on the Torus}

\author{M.~B.~Erdo\smash{\u{g}}an and  N.~Tzirakis}
\thanks{The authors were partially supported by NSF grants DMS-0900865 (B.~E.), and DMS-0901222 (N.~T.) }

\address{Department of Mathematics \\
University of Illinois \\
Urbana, IL 61801, U.S.A.}

\email{berdogan@math.uiuc.edu \\ tzirakis@math.uiuc.edu }

\begin{abstract}
In this paper we consider the Zakharov system with periodic boundary conditions in dimension one. In the first part of the paper, it is shown that for fixed initial data in a Sobolev space, the difference of the nonlinear and the linear evolution is in a smoother space for all times   the solution exists. The smoothing index depends on a parameter distinguishing the resonant and nonresonant cases.  As a corollary, we obtain polynomial-in-time bounds for the   Sobolev norms with regularity above the energy level. In the second part of the paper, we consider the forced and damped Zakharov system and obtain analogous smoothing estimates. As a corollary we prove the existence and smoothness of global attractors in the energy space.
\end{abstract}

\maketitle

\section{Introduction}
The Zakharov system is a system of non-linear partial differential equations, introduced by Zakharov in 1972, \cite{vz}.  It describes the propagation of Langmuir waves in an ionized plasma. The system with periodic boundary conditions consists of a complex field $u$ (Schr\"odinger part) and a real field $n$ (wave part) satisfying the equation:
\begin{equation}\label{eq:zakharov}
\left\{
\begin{array}{l}
iu_{t}+\alpha u_{xx} =nu, \,\,\,\,  x \in {\mathbb T}, \,\,\,\,  t\in [-T,T],\\
n_{tt}-n_{xx}=(|u|^2)_{xx}, \\
u(x,0)=u_0(x)\in H^{s_0}(\mathbb T), \\
n(x,0)=n_0(x)\in H^{s_1}(\mathbb T), \,\,\,\,n_t(x,0)=n_1(x)\in H^{s_1-1}(\mathbb T),
\end{array}
\right.
\end{equation}
where $\alpha>0$  and $T$ is the time of existence of the solutions. The function $u(x,t)$ denotes the slowly varying envelope of the electric field with a prescribed frequency and the function $n(x,t)$ denotes the deviation of the ion density from the equilibrium. Here $\alpha$ is the dispersion coefficient. In the literature (see, e.g., \cite{ht}) it is standard to include the speed of an ion acoustic wave in a plasma as a coefficient $\beta^{-2}$ in front of $n_{tt}$ where $\beta>0$. One can scale away this parameter using time and amplitude coefficients of the form $t\rightarrow \beta t$, $u \rightarrow \sqrt{\beta}u$, and $n \rightarrow \beta n$ and reduce the system to \eqref{eq:zakharov}. Smooth solutions of the Zakharov system obey the following conservation laws:
$$\|u(t)\|_{L^2(\T)}=\|u_{0}\|_{L^2(\T)}$$
and
$$E(u,n,\nu)(t)=\alpha\int_{\T}|\partial_{x}u|^2dx+\frac{1}{2}\int_{\T} n^2dx+\frac{1}{2}\int_{\T}\n^2dx+\int_{\T}n|u|^2dx=E(u_{0},n_{0},n_{1})$$
where $\nu$ is such that $n_t=\nu_{x}$ and $\nu_t=(n+|u|^2)_x$. These conservation laws identify $H^1 \times L^2 \times H^{-1}$ as the energy space for the system.

For $\alpha=1$, Bourgain, in \cite{jbz}, proved that the problem is locally well-posed in the energy space using the restricted norm method (see, e.g., \cite{Bou2}). The solutions are well-posed in the sense of the following definition
\begin{defin} Let $X, Y, Z$ be Banach spaces.   We say that the system of equations \eqref{eq:zakharov}  is locally well-posed in $H^{s_{0}}(\T)\times H^{s_{1}}(\T)\times H^{s_{1}-1}(\T)$, if for a given initial data $(u_{0}, n_{0}, n_{1}) \in H^{s_{0}}(\T)\times H^{s_{1}}(\T)\times H^{s_{1}-1}(\T)$, there exists $T=T(\|u_0\|_{H^{s_{0}}},\|n_0\|_{H^{s_{1}}},\|n_1\|_{H^{s_{1}-1}})>0$ and a unique solution $$(u,n,n_t) \in \left(X \cap C_{t}^{0}H_{x}^{s_0}([-T,T]\times \mathbb T),\, Y \cap C_{t}^{0}H_{x}^{s_1}([-T,T]\times \mathbb T),\, Z \cap C_{t}^{0}H_{x}^{s_1-1}([-T,T]\times \mathbb T)\right) .$$ We also demand that there is continuity with respect to the initial data in the appropriate topology. If $T$ can be taken to be arbitrarily large then we say that the problem is globally well-posed.
\end{defin}
Thus, the energy solutions exist for all times due to the a priori bounds on the local theory norms. We should note that although the quantity $\int_{\T}n|u|^2dx$  has no definite sign it can be controlled using Sobolev inequalities by the $H^1$ norm of $u$ and the $L^2$ norm of $n$, \cite{pesher}. In \cite{ht}, Takaoka extended the local-in-time theory of Bourgain and proved that when $\frac{1}{\alpha} \in \N$ we have local well-posedness in $H^{s_{0}}\times H^{s_{1}}\times H^{s_{1}-1}$ for  $s_1\geq 0$ and $\max(s_1, \frac{s_1}2+\frac12)\leq s_0\leq s_1+1 $. In the case that $\frac{1}{\alpha} \not\in \N$ one has local well-posedness for $s_1\geq -\frac12$, $\max(s_1, \frac{s_1}2+\frac14)\leq s_0\leq s_1+1 $. A recent result, \cite{kishi}, establishes well-posedness in the case of the higher dimensional torus.

The corresponding Cauchy problem on $\R^d$ has a long history. In this case it is somehow easier to establish the well-posedness of the system  due to the dispersive effects of the solution waves.  We cite the following papers \cite{aa1, aa2, bhht, bc, cht, GTV, kpv, ss} as a historical summary of the results. It is expected that (see, e.g., \cite{kishi}) the optimal regularity range for local well-posedness is on the line $s_1=s_0-\frac{1}{2}$ because the two equations in the Zakharov system equally share the loss of derivative. The Zakharov system is not scale invariant but it can be reduced to a simplified system like in \cite{GTV}, and one can then define a critical regularity. This is given by the pair $(s_{0},s_{1})=(\frac{d-3}{2},\frac{d-4}{2})$, which is also on the line. In dimensions $1$ and $2$, the lowest regularity   for the system to have local solutions has been found to be $(s_{0},s_{1})=(0,-\frac{1}{2})$, \cite{GTV}. It is harder to establish the global solu!
 tions at  this level since there is no conservation law controling the wave part. This has been done only in $1$d, \cite{cht}.

In the first part of this paper we study the dynamics of the solutions in more detail. We prove that the difference between the nonlinear and the linear evolution for both the Schr\"odinder and the wave part is in a smoother space than the corresponding initial data, see Theorem~\ref{thm:main1} and Theorem~\ref{thm:main2} below. This smoothing property is not apparent if one views the nonlinear evolution as a perturbation of the linear flow and apply standard Picard iteration techniques to absorb the nonlinear terms. The result will follow from a combination of the method of normal forms (through differentiation by parts) inspired by the result in \cite{bit}, and the restricted norm method of Bourgain, \cite{Bou2}. Here the method is applied to a dispersive system of equations where the resonances are harder to control and the coupling nonlinear terms introduce additional difficulties in estimating the first order corrections. As a corollary, in the case $\alpha>0$, we obtai!
 n   polynomial-in-time bounds for Sobolev norms above the energy level $(s_{0}, s_{1})=(1,0)$ by a bootstrapping argument utilizing the a priori bounds and the smoothing estimates, see Corollary~\ref{cor:growth} below.
We have recently applied this method to obtain similar results for the periodic KdV with a smooth space-time potential, \cite{et}.  For the details of the method the reader can consult \cite{et}.

In the second part we study the existence of a global attractor (see the next section for a definition of global attractors and the statement of our result) for the dissipative Zakharov system in the energy space. Our motivation comes from the smoothing estimates that we obtained in the first part of the paper and our work in \cite{et1}. More precisely we consider
\begin{equation}\label{eq:fdzakharov}
\left\{
\begin{array}{l}
iu_{t}+\alpha u_{xx}+i\gamma u=nu+f, \,\,\,\,  x \in {\mathbb T}, \,\,\,\,  t\in [-T,T],\\
n_{tt}-n_{xx}+\delta n_t =(|u|^2)_{xx}+g, \\
u(x,0)=u_0(x)\in H^{1}(\mathbb T), \\
n(x,0)=n_0(x)\in L^{2}(\mathbb T), \,\,\,\,n_t(x,0)=n_1(x)\in H^{-1}(\mathbb T), \,\,\,\, f \in H^{1}(\T), \,\,\,\, g \in L^{2}(\T)
\end{array}
\right.
\end{equation}
where $f,\ g$ are  time-independent, $g$ is mean-zero,  $\int_{\T}g(x)dx=0$, and the damping coefficients $\delta,\ \gamma >0$.   For simplicity we set $\gamma = \delta$,   and $g=0$. Our calculations apply equally well to the full system and all proofs go through with minor modifications (in particular, one does not need any other a priori estimates).

The problem with Dirichlet boundary conditions has been considered in \cite{fla} and \cite{gm} in more regular spaces than the energy space. The regularity of the attractor in Gevrey spaces with peridic boundary problem was considered in \cite{sch}. 

\subsection{Notation}
To avoid the use of multiple constants, we  write $A \lesssim B$ to denote that there is an absolute  constant $C$ such that $A\leq CB$.  
We also write $A\sim B$ to denote both $A\lesssim B$ and $B \lesssim A$. 
We also define $\langle \cdot\rangle =1+|\cdot|$.  

We define the Fourier sequence of a $2\pi$-periodic $L^2$ function $u$ as
$$u_k=\frac1{2\pi}\int_0^{2\pi} u(x) e^{-ikx} dx, \,\,\,k\in \mathbb Z.$$
With this normalization we have
$$u(x)=\sum_ke^{ikx}u_k,\,\,\text{ and } (uv)_k=u_k*v_k=\sum_{m+n=k} u_nv_m.$$

As usual, for $s<0$, $H^{s}$ is the completion of $L^2$ under the norm
$$\|u\|_{H^{s}}=\|\widehat u(k) \langle k\rangle^{s}\|_{\ell^2}.$$
Note that for a mean-zero $L^2$ function $u$, $\|u\|_{H^{s}}\sim \|\widehat u(k) |k|^{s}\|_{\ell^2}$.
For a sequence $u_k$, with $u_0=0$, we will use $\|u\|_{H^{s}}$ notation to denote $\|u_k |k|^s\|_{\ell^2}$.
We also define $\dot H^s=\{u\in H^s: u \text{ is mean-zero}\}$.

The following function will appear many times in the proofs below.
\be\nn
\phi_\beta(k):=\sum_{|n|\leq |k|}\frac1{|n|^\beta}\sim \left\{\begin{array}{ll}
1, & \beta>1,\\
\log(1+\la k\ra), &\beta=1,\\
\la k \ra^{1-\beta}, & \beta<1.
 \end{array}\right.
\ee

\section{Statement of Results}
\subsection{Smoothing Estimates for the Zakharov System}
First note that if $n_0$ and $n_1$ are mean-zero then
$n$, $n_t$ remain mean-zero during the evolution since by integrating the wave part of the system we obtain $\partial_{t}^{2}\int_{\T}n(x,t)dx=0$. We will work with this mean-zero assumption in this paper. This is no loss of generality since if $\int_{T}n_{0}(x)dx=A$ and $\int_{T}n_{1}(x)dx=B$, then one can consider the new variables $n \rightarrow n-A-Bt$ and $u \rightarrow e^{i(B\frac{t^2}{2}+At)}u$, and obtain the same system with  mean-zero data.

By considering the operator $d=(-\partial_{xx})^{1/2}$, and writing $n_\pm=n\pm i d^{-1}n_t$, the system can be rewritten as
\begin{equation}\label{eq:zakharov1}
\left\{
\begin{array}{l}
 iu_{t}+\alpha u_{xx} =\frac12(n_++n_-)u, \,\,\,\,  x \in {\mathbb T}, \,\,\,\,  t\in [-T,T],\\
 (i\partial_t\mp d)n_\pm=\pm d (|u|^2), \\
 u(x,0)=u_0(x)\in H^{s_0}(\mathbb T), \,\,\,\, 
n_\pm(x,0)=n_0(x)\pm i d^{-1} n_1(x) \in H^{s_1}(\mathbb T).
\end{array}
\right.
\end{equation}
Note that $d^{-1} n_1(x)$ is well-defined because of the mean-zero assumption, and that $n_+=\overline{n_-}$.

The local well posedness of the system was established in the framework of $X^{s,b}$ spaces introduced by Bourgain in \cite{Bou2}. Let
$$
\|u\|_{X^{s,b}}=\big\|\la k\ra^s \la \tau-\alpha k^2\ra^b \widehat u(k,\tau)  \big\|_{\ell^2_k L^2_\tau},
$$
$$
\|n\|_{Y_\pm^{s,b}}=\big\|\la k\ra^s \la \tau\mp |k|\ra^b \widehat n(k,\tau)  \big\|_{\ell^2_k L^2_\tau}.
$$
Here `$\pm$' corresponds to the norm of $n_\pm$ in the system \eqref{eq:zakharov1}. As usual we also define the restricted norm
$$
\|u\|_{X^{s,b}_T}=\inf_{\widetilde u=u, \,t\in [-T,T]} \|\widetilde u\|_{X^{s,b}}.
$$
The norms $Y_{\pm,T}^{s,b}$ are defined accordingly. We also abbreviate $n_{\pm}(x,0)=n_{\pm,0}$.

\begin{defin}
We say $(s_0,s_1)$ is $\alpha$-admissable if $s_1\geq -\frac12$ and $\max(s_1, \frac{s_1}2+\frac14)\leq s_0\leq s_1+1 $ for $\frac1\alpha\not\in \N$, or if $s_1\geq 0$ and $\max(s_1, \frac{s_1}2+\frac12)\leq s_0\leq s_1+1 $ for $\frac1\alpha\in \N$.
\end{defin}

Takoaka's theorem on local well-posedness can be stated as
\begin{theorem}\cite{ht}\label{thm:takaoka} Suppose $\alpha\neq 0$ and  $(s_0,s_1)$ is $\alpha$-admissable.
 Then given initial data $(u_0,n_{+,0},n_{-,0})\in H^{s_0}\times H^{s_1} \times H^{s_1}$ there exists
$$T\gtrsim \big(\| u_0\|_{H^{s_0}}+\|n_{+,0}\|_{ H^{s_1}}+ \|n_{-,0}\|_{ H^{s_1}} \big)^{-\frac1{12}+},$$
and a unique solution $(u,n_+,n_-)\in C\big([-T,T]:H^{s_0}\times H^{s_1} \times H^{s_1}\big)$. Moreover, we have
$$
\|u\|_{X^{s,\frac12}_T}+\|n_{+,0}\|_{ Y_{+,T}^{s_1,\frac12}}+ \|n_{-,0}\|_{Y_{-,T}^{s_1,\frac12}}
\leq 2 \big(\| u_0\|_{H^{s_0}}+\|n_{+,0}\|_{ H^{s_1}}+ \|n_{-,0}\|_{ H^{s_1}} \big).
$$
\end{theorem}
Now, we can state our results on the smoothing estimates:
\begin{theorem}  \label{thm:main1}
Suppose $\frac1\alpha\not\in \N$, and $(s_0,s_1)$ is $\alpha$-admissable.
Consider the solution of \eqref{eq:zakharov1} with initial data $(u_0,n_{+,0},n_{-,0})\in H^{s_0}\times H^{s_1} \times H^{s_1}$.
Assume that we have a growth bound
$$\|u(t)\|_{H^{s_0}}+\|n_{+}(t)\|_{ H^{s_1}}+ \|n_{-}(t)\|_{ H^{s_1}}\leq C\big(\| u_0\|_{H^{s_0}}+\|n_{+,0}\|_{ H^{s_1}}+ \|n_{-,0}\|_{ H^{s_1}} \big) (1+|t|)^{\gamma(s_0,s_1)}.$$
Then, for any $a_0\leq \min(1,2s_0,1+2s_1)$ (the inequality has to be strict if $s_0-s_1=1$) and for any $a_1\leq \min(1,2s_0,2s_0-s_1)$, we have
\begin{align}\label{usmooth}
u(t)-e^{i\alpha t\partial_x^2}u_0 &\in C^0_tH_{x}^{s_0+a_0}(\R\times \T  ),\\ n_\pm(t)-e^{\mp itd}n_{\pm,0}  &\in C^0_tH_{x}^{s_1+a_1}(\R\times \T  ).\label{nsmooth}
\end{align}
Moreover, for $\beta>1+15\gamma(s_0,s_1)$, we have
\begin{align}\label{growth}
\|u(t)-e^{i\alpha t\partial_x^2}u_0\|_{ H^{s_0+a_0}} +\| n_\pm(t)-e^{\mp itd}n_{\pm,0}\|_{H^{s_1+a_1}} \leq C   (1+|t|)^{\beta},
\end{align}
where $C=C\big(s_0,s_1,a_0,a_1,\| u_0\|_{H^{s_0}},\|n_{+,0}\|_{ H^{s_1}}, \|n_{-,0}\|_{ H^{s_1}} \big).$
\end{theorem}
\begin{theorem}  \label{thm:main2}
Suppose $\frac1\alpha \in \N$, and $(s_0,s_1)$ is $\alpha$-admissable.
Assume that we have a growth bound
$$\|u(t)\|_{H^{s_0}}+\|n_{+}(t)\|_{ H^{s_1}}+ \|n_{-}(t)\|_{ H^{s_1}}\leq C\big(\| u_0\|_{H^{s_0}}+\|n_{+,0}\|_{ H^{s_1}}+ \|n_{-,0}\|_{ H^{s_1}} \big) (1+|t|)^{\alpha(s_0,s_1)}.$$
Then, for any $a_0\leq \min(1, s_1)$ (the inequality has to be strict if $s_0-s_1=1$ and $s_1\geq 1$) and for any $a_1\leq \min(1, 2s_0-s_1-1)$, we have \eqref{usmooth}, \eqref{nsmooth} and \eqref{growth}.
\end{theorem}

As an application of these theorems we obtain the following corollary regarding the growth of higher order
Sobolev norms.
\begin{corollary}\label{cor:growth} For any $\alpha>0$, and for any $\alpha$-admissable $(s_0,s_1)$ with $s_0\geq 1$, $s_1\geq 0$, the global solution of \eqref{eq:zakharov1} with
$H^{s_0}\times H^{s_1}\times H^{s_1}$ data satisfies the growth bound
$$
\| u(t)\|_{H^{s_0}}+\|n_{+}(t)\|_{ H^{s_1}}+ \|n_{-}(t)\|_{ H^{s_1}} \leq C_1 (1+|t|)^{C_2},
$$
where $C_1$  depends on  $s_0,s_1$, and $\| u_0\|_{H^{s_0}}+\|n_{+,0}\|_{ H^{s_1}}+ \|n_{-,0}\|_{ H^{s_1}}$, and  $C_2$  depends on $s_0,s_1$.

\end{corollary}
\begin{proof}
We drop `$\pm$' signs and work with $u$ and $n$.
First note that because of the energy conservation, $\|u\|_{H^1}$ and $\|n\|_{L^2}$ are bounded for all times. Assume that the claim holds for regularity levels $(s_0,s_1)$. Let $(a_0,a_1)$ be given by Theeorem~\ref{thm:main1} or Theorem~\ref{thm:main2}. Note that for initial data in $H^{s_0+a_0}\times H^{s_1+a_1}$, applying the theorem with $(s_0,s_1)$ and $(a_0,a_1)$, we have
$$
\|u(t)-e^{i\alpha t\partial_x^2}u_0\|_{ H^{s_0+a_0}} +\| n_\pm(t)-e^{\mp itd}n_{\pm,0}\|_{H^{s_1+a_1}} \leq C   (1+|t|)^{\beta}.
$$
Therefore, since the linear groups are unitary, we have
$$
\| u(t)\|_{H^{s_0+a_0}}+\|n(t)\|_{ H^{s_1+a_1}} \leq  C   (1+|t|)^{\beta}+\|u_0\|_{H^{s_0+a_0}}+\|n_0\|_{H^{s_1+a_0}}.
$$
The statement follows by induction on the regularity.

We note that in the case $\frac1\alpha\in\N$, $s_0=1$, $s_1=0$, we have $a_0=0$. However, since
$a_1\in[0,1]$, we obtain the statement for $\alpha$-admissable $(1,s_1)$, $0\leq s_1\leq 1$. From then on
we can take both $a_0>0$ and $a_1>0$.
\end{proof}

\subsection{Existence of a Global Attractor for the Dissipative Zakharov System}

The problem of global attractors for nonlinear PDEs is concerned with the description of the nonlinear dynamics for a given problem as $t \to \infty$. In particular assuming that one has a well-posed problem for all times we can define the semigroup operator $U(t):u_{0}\in H\to u(t)\in H$ where $H$ is the phase space. We want to describe the long time asymptotics of the solution by an invariant set  $X\subset H$ (a global attractor) to which the orbit converges as $t \to \infty$:
$$U(t)X=X, \ \ t\in \R_+,\,\,\,\,\,\,\,\,\,\,d(u(t), X)\to 0.$$
For dissipative systems there are many results (see, e.g., \cite{temam}) establishing the existence of  a  compact   set that satisfies the above properties. Dissipativity is characterized by the existence of a bounded absorbing set into which all solutions enter eventually.  The candidate for the attractor set  is the omega limit set of an absorbing set, $B$, defined by $$\omega(B)=\bigcap_{s \geq 0}\overline{\bigcup_{t\geq s}U(t)B}$$
where the closure is taken on $H$. To state our result we need some definitions from \cite{temam} (also see \cite{et1} for more discussion).
\begin{defin}
We say that a compact subset $\mathcal{A}$ of $ H$ is a  global attractor  for the semigroup $\{U(t)\}_{t \geq 0}$ if $\mathcal{A}$ is invariant under the flow and if  for every $u_{0}\in H$, $d(U(t)u_{0},\mathcal{A})\to 0$ as $t\to \infty$.
\end{defin}
The distance is understood to be the distance of a point to the set $d(x,Y)=\inf_{y\in Y}d(x,y)$.

To state a general theorem for the existence of a global attractor we need one more definition:
\begin{defin}
We say a bounded subset  $\mathcal{B}_0$   of $H$   is  absorbing  if   for any bounded $\mathcal{B} \subset H$ there exists $T=T(\mathcal{B})$ such that for all $t \geq T$, $U(t)\mathcal{B} \subset \mathcal{B}_0$.
\end{defin}
It is not hard to see that the existence of a global attractor $\mathcal{A}$ for a semigroup $U(t)$ implies the existence of an absorbing set. For the converse  we cite the following theorem from \cite{temam} which gives a  general criterion for the existence of a global attractor.

\vspace{5mm}

\noindent
{\bf Theorem A.} {\it
We assume that $H$ is a metric space and that the operator $U(t)$ is a continuous semigroup from $H$ to itself for all $t \geq 0$. We also assume that there exists an  absorbing set $\mathcal{B}_0$. If the semigroup $\{U(t)\}_{t\geq 0}$ is asymptotically compact, i.e. for every bounded sequence $x_k$ in $H$ and every sequence $t_k \to \infty$, $\{U(t_k)x_k\}_{k}$ is relatively compact in $H$, then  $\omega( \mathcal{B}_0)$ is a global attractor.}

Using Theorem A and a smoothing estimate as above, we will prove the following
\begin{theorem}\label{thm:attractor}
Fix $\alpha>0$.
Consider the  dissipative Zakharov system   \eqref{eq:fdzakharov}
on $\mathbb{T}\times [0,\infty)$ with  $u_0\in   H^1$ and with mean-zero $n_0\in L^2$, $n_1\in H^{-1}$. Then the equation possesses a global attractor in $H^1\times \dot L^2\times \dot H^{-1}$. Moreover, for any $a\in (0,1)$,  the global attractor is a compact subset of $H^{1+a}\times H^{a}\times H^{-1+a}$, and it is bounded in $H^{1+a}\times H^{a}\times H^{-1+a}$ by a constant depending  only on $a, \alpha, \gamma$, and $\|f\|_{H^1}$.
\end{theorem}

To prove Theorem~\ref{thm:attractor}  in the case $\frac1\alpha\not \in \N$ we will demonstrate  that  the solution decomposes into two parts; a linear one which decays to zero as time goes to infinity and a nonlinear one which always belongs to a smoother space. As a corollary we prove that all solutions are attracted by a ball in $H^{1+a}\times H^{a}\times H^{-1+a}$, $a\in(0,1)$, whose radius depends only on $a$, the $H^1$ norm of the forcing term and the damping parameter. This implies the existence of a smooth global attractor and provides quantitative information on the size of the attractor set in $H^{1+a}\times H^{a}\times H^{-1+a}$. In addition it implies  that higher order Sobolev norms are bounded for all positive times, see \cite{et1}. In the case $\frac1\alpha\in \N$ the proof is slightly different because of a resonant term.

We close this section with a discussion of the  well-posedness of \eqref{eq:fdzakharov} in $H^1\times L^2\times H^{-1}$. We first rewrite the system (when $\gamma=\delta, g=0$) by passing to $n_\pm$ variables as above:
\begin{equation}\label{eq:fdzakharov1}
\left\{
\begin{array}{l}
 (i\partial_t+\alpha\partial_x^2+i\gamma)u =\frac{n_++n_-}{2}  u+ f , \,\,\,\,  x \in {\mathbb T}, \,\,\,\,  t\in [-T,T],\\
 (i\partial_t\mp  d+i\gamma)n_\pm = \pm d (|u|^2), \\
 u(x,0)=u_0(x) \in H^{1}(\mathbb T), \,\,\,\, 
n_\pm(x,0)=n_{\pm,0}(x)=n_0(x)\pm i d^{-1} n_1(x) \in L^2(\mathbb T).
\end{array}
\right.
\end{equation}

\begin{theorem}\label{thm:fdlocal}
Given initial data $(u_0,n_{+,0},n_{-,0})\in H^{1}\times L^2 \times L^2$ there exists
$$T=T\big(\| u_0\|_{H^{1}},\|n_{+,0}\|_{ L^2}, \|n_{-,0}\|_{ L^2},\|f\|_{H^1},\gamma \big),$$
and a unique solution $(u,n_+,n_-)\in C\big([-T,T]:H^{1}\times L^2 \times L^2\big)$ of \eqref{eq:fdzakharov1}. Moreover, we have
$$
\|u\|_{X^{1,\frac12}_T}+\|n_{+,0}\|_{ Y_{+,T}^{0,\frac12}}+ \|n_{-,0}\|_{Y_{-,T}^{0,\frac12}}
\leq 2 \big(\| u_0\|_{H^{1}}+\|n_{+,0}\|_{ L^2}+ \|n_{-,0}\|_{L^2} \big).
$$
\end{theorem}
This theorem follows by using the a priori estimates of Takaoka in \cite{ht}. In the case of forced and damped KdV, this was done in \cite[Theorem 2.1, Lemma 2.2]{et1}.

  The global well-posedness follows from the following a priori estimate for the system \eqref{eq:fdzakharov1} which was obtained in \cite{fla} (recall that $n_\pm=n\pm id^{-1}n_t$): 
\be\label{aprioribd}
\| u \|_{H^{1}}+\|n_{+ }\|_{ L^2}+ \|n_{- }\|_{L^2} \leq C_1+C_2e^{-C_3t}, \,\,\,\,\,t>0,
\ee
where $C_1=C_1(\alpha,\gamma,\|f\|_{H^{1}})$, $C_2=C_2(\alpha,\gamma,\|f\|_{H^{1}},\| u_0\|_{H^{1}},\|n_{\pm,0}\|_{ L^2} )$, and $C_3=C_3(\alpha,\gamma)$.
In fact this was proved in \cite{fla} for   Dirichlet boundary conditions. In the case of periodic boundary conditions, the proof remains valid. Note that \eqref{aprioribd} also implies the existence of an absorbing set $\mathcal{B}_0$ in $H^{1}\times L^2 \times L^2$ of radius $C_1(\alpha,\gamma,\|f\|_{H^{1}})$.

\section{Proof of Theorem~\ref{thm:main1} and Theorem~\ref{thm:main2}}\label{sec:evol}
In this section we drop the `$\pm$' signs and work with one $n$. We also set $Y=Y_+$.
\begin{equation}\label{eq:zakharov2}
\left\{
\begin{array}{l}
 iu_{t}+\alpha u_{xx} =n u, \,\,\,\,  x \in {\mathbb T}, \,\,\,\,  t\in [-T,T],\\
 (i\partial_t- d)n =  d (|u|^2), \\
 u(x,0)=u_0(x)\in H^{s_0}(\mathbb T), \,\,\,\, 
n(x,0)=n_0(x)+ i d^{-1} n_1(x) \in H^{s_1}(\mathbb T).
\end{array}
\right.
\end{equation}

\begin{rmk}
We note that since $n_+=\overline{n_-}$ all of our claims about \eqref{eq:zakharov2} is also valid for \eqref{eq:zakharov1}. The difference in the proof will arise in the differentiation by parts process and the $X^{s,b}$ estimates. Because of $\eqref{pmissue}$, in the formulas \eqref{veqnmmv}, \eqref{meqnmvv}, \eqref{meqnvmv}, there will additional sums in which every term, in the phase and in the multiplier with an $|\cdot|$ sign, will have a `$\pm$' sign  in front. This change won't alter the proofs for the $X^{s,b}$ estimates, in fact, all the cases we considered will work exactly the same way. Also it won't change the structure of the resonant sets in the case $\frac1\alpha\in\N$.
\end{rmk}

We will prove Theorem~\ref{thm:main2} only for $\alpha=1$. Therefore, below we either have $\frac1\alpha\not \in \N$ or $\alpha=1$. The case $\alpha\neq 1, \frac1\alpha\in \N$ can be handled by only cosmetic changes in the proof. Writing
$$u(x,t)=\sum_{k } u_k(t) e^{ikx},\,\,\,\,\,n(x,t)=\sum_{j\neq 0} n_j(t) e^{ijx},$$
we obtain the following system for the Fourier coefficients:
\begin{equation}\label{eq:zakharov3}
\left\{
\begin{array}{l}
 i\partial_t  u_k - \alpha k^2 u_k =\sum_{k_1+k_2=k,\,k_1\neq 0} n_{k_1} u_{k_2}, \\
 i\partial_t n_j- |j| n_j = |j| \sum_{j_1+j_2=j} u_{j_1} \overline{u_{-j_2}}, \,\,\,\,j\neq 0 \\
 u_k(0)=(u_0)_k, \,\,\,\, 
n_j(0)=(n_0)_j+ i |j|^{-1} (n_1)_j,\,\,\,j\neq 0.
\end{array}
\right.
\end{equation}

We start with the following proposition which follows from differention by parts.

\begin{prop}\label{thm:dbp}
The system \eqref{eq:zakharov3} can be written in the following form:
\begin{align}\label{v_eq_dbp}
i\partial_t\big[e^{it\alpha k^2}u_k+e^{it\alpha k^2}B_1(n,u)_k\big]&= e^{it\alpha k^2}\big[\rho_1(k)+R_1(u)(\widehat k, t)+R_2(u,n)(\widehat k, t)\big],\\
 i\partial_t \big[e^{it|j|}n_j+e^{it|j|} B_2(u)_j\big] &= e^{it|j|}\big[\rho_2(j)+R_3(u,n)(\widehat j, t)+R_4(u,n)(\widehat j, t)\big], \label{m_eq_dbp}
\end{align}
where
\begin{align*}
B_1(n,u)_k=\sum^*_{k_1+k_2=k,k_1\neq 0}\frac{  n_{k_1}u_{k_2}}{\alpha k^2-\alpha k_2^2-|k_1|},\,\,\,\,\,\,\,
B_2(u)_j=|j| \sum^*_{j_1+j_2=j}\frac{   u_{j_1} \overline{u_{-j_2}}
 }{ |j|-\alpha j_1^2+\alpha j_2^2 }.
\end{align*}
$$
R_1(u)(\widehat k, t)=\sum_{k_1,k_2}^* \frac{|k_1+k_2| u_{k_1}\overline{u_{-k_2}}\, u_{k-k_1-k_2}}{\alpha k^2-\alpha (k-k_1-k_2)^2-|k_1+k_2|}.
$$
$$
R_2(u,n)(\widehat k, t)=\sum_{k_1,k_2\neq 0}^* \frac{n_{k_1}n_{k_2} u_{k-k_1-k_2}}{ \alpha k^2-\alpha (k-k_1)^2 -|k_1| }.
$$
$$
R_3(u,n)(\widehat j, t)=|j| \sum_{j_1\neq 0, j_2}^* \frac{   n_{j_1} u_{j_2} \overline{u_{j_1+j_2-j}}
 }{ |j|-\alpha (j_1+j_2)^2+\alpha (j-j_1-j_2)^2 }.
$$
$$
R_4(u,n)(\widehat j, t)=|j| \sum_{j_1\neq 0, j_2}^* \frac{ \overline{n_{-j_1}}  u_{j_2}   \, \overline{u_{j_1+j_2-j}}
 }{ |j|-\alpha j_2^2+\alpha (j-j_2)^2 }.
$$
Here, $\sum^*$ means that the sum is over all nonresonant terms, i.e., over all indices for which the denominator is not zero.
Moreover, the resonant terms $\rho_1$ and $\rho_2$ are zero if $\frac1\alpha \not \in\N$. For $\alpha=1$,
$$
\rho_1(k)=n_{2k-\sgn(k)}u_{\sgn(k)-k},\,\,k\neq 0,\,\,\,\,\,\,\,\,\,\,\,\rho_2(j)=|j|u_{\frac{j+\sgn(j)}{2}}\overline{u_{\frac{j-\sgn(j)}{2}}},\,\,\,j \text{ odd}.
$$
\end{prop}
\begin{proof}[Proof of Proposition~\ref{thm:dbp}]
Changing the variables $m_j=n_je^{i|j|t}$ and $v_k=u_ke^{i\alpha k^2 t}$ in \eqref{eq:zakharov3}, we obtain
\begin{equation}\label{eq:zakharov4}
\left\{
\begin{array}{l}
 i\partial_t  v_k   =\sum_{k_1+k_2=k,\,k_1\neq 0} e^{it(\alpha k^2-\alpha k_2^2 -|k_1|)} m_{k_1} v_{k_2}, \\
 i\partial_t m_j  = |j| \sum_{j_1+j_2=j} e^{it(|j|-\alpha j_1^2+\alpha j_2^2)} v_{j_1} \overline{v_{-j_2}}, \,\,\,\,j\neq 0 \\
 v_k(0)=(u_0)_k, \,\,\,\, 
m_j(0)=(n_0)_j+ i |j|^{-1} (n_1)_j,\,\,\,j\neq 0.
\end{array}
\right.
\end{equation}
It is easy to check that if we define $m_j^+$ and $m_j^-$ accordingly, then
\be\label{pmissue}
\partial_t m_j^-=\overline{\partial_t m_{-j}^+}.
\ee

Note that the exponents do not vanish if $1/\alpha$ is not an integer.
On the other hand if $\alpha=1$, then the resonant set is:
\begin{align*}
&(k_1,k_2)=\big(2k-\sgn(k),\sgn(k)-k\big),\,\,\,k\neq 0.\\
&(j_1,j_2)=\Big(\frac{j+\sgn(j)}{2},\frac{j-\sgn(j)}{2}\Big),\,\,\,j \text{ odd}.
\end{align*}
The contribution of the corresponding terms give $\rho_1$ and $\rho_2$ in the case $\alpha=1$. Below, we assume that $\frac1\alpha\not\in\N$.

Differentiating by parts in the $v$ equation we obtain
\begin{align*}
 i\partial_t  v_k    =\sum_{k_1+k_2=k,\,k_1\neq 0} e^{it(\alpha k^2-\alpha k_2^2 -|k_1|)} m_{k_1} v_{k_2}
&=\sum_{k_1+k_2=k,\,k_1\neq 0} \frac{\partial_t \big(e^{it(\alpha k^2-\alpha k_2^2 -|k_1|)} m_{k_1} v_{k_2}\big)}{i(\alpha k^2-\alpha k_2^2 -|k_1|)}\\
&+i\sum_{k_1+k_2=k,\,k_1\neq 0} \frac{ e^{it(\alpha k^2-\alpha k_2^2 -|k_1|)} \partial_t\big(m_{k_1} v_{k_2}\big)}{ \alpha k^2-\alpha k_2^2 -|k_1| }.
\end{align*}
The second sum can be rewritten using the equation as follows:
\begin{align}\nn
&\sum_{k_1+k_2+k_3=k,\,k_1+k_2\neq 0} \frac{ e^{it\alpha( k^2-  k_1^2+  k_2^2-  k_3^2  )}    |k_1+k_2| v_{k_1} \overline{v_{-k_2}}  v_{k_3} }{ \alpha k^2-\alpha k_3^2 -|k_1+k_2| }\\
&+\sum_{k_1+k_2+k_3=k,\,k_1,k_2\neq 0} \frac{ e^{it(\alpha k^2-\alpha k_3^2  -|k_1|-|k_2|)}
    m_{k_1}m_{k_2} v_{k_3}
 }{ \alpha k^2-\alpha (k_2+k_3)^2 -|k_1| }.\label{veqnmmv}
\end{align}

Now, we differentiate by parts in the $m$ equation:
 \begin{align*}
 i\partial_t  m_j =|j| \sum_{j_1+j_2=j} e^{it(|j|-\alpha j_1^2+\alpha j_2^2)} v_{j_1} \overline{v_{-j_2}}
&=|j| \sum_{j_1+j_2=j}\frac{\partial_t \big(e^{it(|j|-\alpha j_1^2+\alpha j_2^2)} v_{j_1} \overline{v_{-j_2}}
\big)}{i(|j|-\alpha j_1^2+\alpha j_2^2)}
   \\
& +i|j| \sum_{j_1+j_2=j}\frac{e^{it(|j|-\alpha j_1^2+\alpha j_2^2)} \partial_t \big(v_{j_1} \overline{v_{-j_2}}
\big)}{ |j|-\alpha j_1^2+\alpha j_2^2 }.
\end{align*}
The second sum can be rewritten using the equation as follows:
\begin{align}\label{meqnmvv}
 &|j| \sum_{j_1+j_2+j_3=j,j_1\neq 0}\frac{e^{it(|j| +\alpha j_3^2-\alpha j_2^2 -|j_1|)}   m_{j_1} v_{j_2} \overline{v_{-j_3}}
 }{ |j|-\alpha (j_1+j_2)^2+\alpha j_3^2 }\\
 &+ |j| \sum_{j_1+j_2+j_3=j, j_2\neq 0}\frac{e^{it(|j|-\alpha j_1^2+\alpha j_3^2+|j_2|)} v_{j_1}   \overline{m_{-j_2}} \, \overline{v_{-j_3}}
 }{ |j|-\alpha j_1^2+\alpha (j_2+j_3)^2 }.\label{meqnvmv}
\end{align}
The statement follows by going back to $u$ and $n$ variables.

\end{proof}

Integrating \eqref{v_eq_dbp} and \eqref{m_eq_dbp} from $0$ to $t$, we obtain
\begin{multline}\label{new_u}
u_k(t)-e^{-it\alpha k^2}u_k(0) =e^{-it\alpha k^2} B_1(n,u)_k(0)-B_1(n,u)_k(t)\\-i\int_0^t e^{-i\alpha k^2(t-s)}\big[\rho_1(k)+R_1(u)(\widehat k, s)+R_2(u,n)(\widehat k, s)\big] ds.
\end{multline}
\begin{multline}\label{new_n}
n_j(t)-e^{-it|j|}n_j(0) =e^{-it|j|} B_2(u)_j(0)-B_2(u)_j(t)\\-i\int_0^t e^{-i|j|(t-s)}\big[\rho_2(j)+R_3(u,n)(\widehat j, s)+R_4(u,n)(\widehat j, s)\big] ds.
\end{multline}
Below we obtain a priori estimates for  $\rho_1,\rho_2,B_1$, and $B_2$. Before that we state a technical lemma that will be used many times in the proofs.
\begin{lemma}\label{lem:sums} a) If  $\beta\geq \gamma\geq 0$ and $\beta+\gamma>1$, then
\be\nn
\sum_n\frac{1}{\la n-k_1\ra^\beta \la n-k_2\ra^\gamma}\lesssim \la k_1-k_2\ra^{-\gamma} \phi_\beta(k_1-k_2).
\ee
b) For $\beta\in(0,1]$, we have
$$
 \int_\R \frac{d\tau}{\la \tau+\rho_1 \ra^\beta \la \tau+\rho_2\ra}\lesssim \frac{1}{\la \rho_1-\rho_2 \ra^{\beta-}}.
$$
c) If $\beta>1/2$, then
$$\sum_n\frac{1}{\la n^2+c_1n+c_2\ra^\beta}\lesssim 1,$$
where the implicit constant is independent of $c_1$ and $c_2$.
\end{lemma}

We will prove this lemma in a appendix.

\begin{lemma}\label{apriori} Under the conditions of Theorem~\ref{thm:main1} and Theorem~\ref{thm:main2}, we have
$$
\|\rho_1\|_{H^s}\les \|n\|_{H^{s_1}}\|u\|_{H^{s_0}}, \,\,\,\,\,\text{ if } s\leq s_0+s_1,
$$
$$
\|\rho_2\|_{H^s}\les  \|u\|_{H^{s_0}}^2, \,\,\,\,\,\text{ if } s\leq 2s_0-1,
$$
$$
\|B_1(n,u)\|_{H^s}\les \|n\|_{H^{s_1}}\|u\|_{H^{s_0}}, \,\,\,\,\,\text{ if } s\leq 1+s_0+\min(s_1,0),
$$
$$
\|B_2(u)\|_{H^s}\les  \|u\|_{H^{s_0}}^2, \,\,\,\,\,\text{ if } s\leq  \min(2s_0,1+s_0).
$$ 
\end{lemma}
\begin{proof}
The proof for $\rho_1$ and $\rho_2$ is immediate from their definition.

To estimate   $B_1$, first note that
$$
\big|\alpha k^2-\alpha k_2^2-|k_1|\big|=|\alpha| |k_1| |2k-k_1-\frac1\alpha \sgn(k_1)|\sim \la k_1\ra \la 2k-k_1\ra.
$$
The last equality is immediate in the case $\frac1\alpha\not\in\N$, when $\alpha =1$, it follows from the nonresonant condition. Therefore we have
$$
|B_1(n,u)_k|\les \sum_{k_1\neq 0}\frac{|n_{k_1}| |u_{k-k_1}|}{\la k_1\ra \la 2k-k_1\ra}.
$$
We estimate the $H^s$ norm as follows:
$$
\|B_1\|_{H^s}^2\lesssim \Big\|\sum_{k_1\neq 0} \la k_1\ra^{2s_1}|n_{k_1}|^2 \la k-k_1\ra^{2s_0}|u_{k-k_1}|^2\Big\|_{\ell^1_k}  \Big\|\sum_{k_1} \frac{\la k\ra^{2s}}{ \la k_1\ra^{2+2s_1}  \la k-k_1\ra^{2s_0}\la 2k-k_1\ra^2}\Big\|_{\ell^\infty_k}
$$
The first sum is bounded by $\|n\|_{H^{s_1}}^2 \|u\|_{H^{s_0}}^2$ since it is a convolution of two $\ell^1$ sequences. To estimate the second sum we distinguish the cases $|k_1|<|k|/2$, $|k_1|>4|k|$, and $|k_1|\sim|k|$. In the first case, we bound the sum by
$$
\sum_{k_1} \frac{\la k\ra^{2s-2-2s_0}}{ \la k_1\ra^{2+2s_1} }\lesssim \la k\ra^{2s-2-2s_0},
$$
since $2+2s_1>1$.
In the second case, we bound the sum by
$$
\sum_{|k_1|>4|k|} \frac{\la k\ra^{2s}}{ \la k_1\ra^{4+2s_1+2s_0} }\lesssim \la k\ra^{2s-3-2s_1-2s_0}\leq \la k\ra^{2s-2-2s_0}.
$$
In the final case, we have
$$
\sum_{|k_1|\sim|k|} \frac{\la k\ra^{2s-2-2s_1}}{\la k-k_1\ra^{2s_0}\la 2k-k_1\ra^2} \lesssim \la k\ra^{2s-2-2s_1-2\min(s_0,1)}.
$$
In the last inequality we used part a) of Lemma~\ref{lem:sums}.

Combining these cases we see that $B_1\in H^s$ for $s\leq 1+\min(s_0,s_1 +\min(s_0,1))$. In particular,
$B_1\in H^s$ if $s\leq 1+s_0+\min(s_1,0)$ which can be seen by distinguishing the cases $s_0\geq 1$ and $s_0<1$ and using the condition $1+s_1\geq s_0$.

Similarly,  we estimate
$$
|B_2(u)_j|\les   \sum_{j_1 }\frac{  |u_{j_1}| |u_{j_1-j}|}{ \la j-2j_1\ra}.
$$
As in the case of $B_1$, we see that $B_2\in H^s$ if
$$
\sup_j \sum_{j_1}\frac{\la j\ra^{2s}}{ \la j-2j_1\ra^2 \la j_1\ra^{2s_0} \la j-j_1\ra^{2 s_0}}<\infty.
$$
We distinguish the cases $|j_1|<|j|/4$, $|j_1|>2|j|$, and $|j_1|\sim|j|$. In the first case, we bound the sum by
$$
\sum_{|j_1|< |j|/4} \frac{\la j\ra^{2s-2-2s_0}}{ \la j_1\ra^{2s_0} }\lesssim \la j\ra^{2s-2-2s_0 }\phi_{2s_0}(j).
$$
In the second case, we bound the sum by
$$
\sum_{|j_1|>2|j|} \frac{\la j\ra^{2s}}{ \la j_1\ra^{2+4s_0} }\lesssim \la j\ra^{2s-1 -4s_0}.
$$
In the final case, we have
$$
\sum_{|j_1|\sim|j|} \frac{\la j\ra^{2s- 2s_0}}{\la j-2j_1\ra^{2}\la j-j_1\ra^{2s_0}} \lesssim \la j\ra^{2s-2s_0-2\min(s_0,1)}.
$$
Combining this cases, we see that $B_2$ is in $H^s$ if $s\leq \min(2s_0,1+s_0)$.
\end{proof}

Using the estimates in Lemma~\ref{apriori} in the equations \eqref{new_u} and \eqref{new_n} after writing the equations in the space side, we obtain
\begin{multline}\label{udelta}
\|u(t)-e^{it\alpha \partial_x^2 }u_0\|_{H^{s_0+a_0}}\lesssim  \|n_0\|_{H^{s_1}}\|u_0\|_{H^{s_0}}+
\|n(t)\|_{H^{s_1}}\|u(t)\|_{H^{s_0}}\\ + \int_0^t \|n(s)\|_{H^{s_1}}\|u(s)\|_{H^{s_0}} ds + \Big\|\int_0^t e^{i\alpha (t-s)\partial_x^2}\big[R_1(u)(s)+R_2(u,n)(s)\big] ds\Big\|_{H^{s_0+a_0}},
\end{multline}
\begin{multline}\label{ndelta}
\|n(t)-e^{-itd}n_0\|_{H^{s_1+a_1}} \les \|u_0\|_{H^{s_0}}^2+\|u(t)\|_{H^{s_0}}^2 \\ + \int_0^t\|u(s)\|_{H^{s_0}}^2 ds  + \Big\|\int_0^t e^{-id(t-s)}\big[R_3(u,n)(s)+R_4(u,n)(s)\big] ds\Big\|_{H^{s_1+a_1}},
\end{multline}
where
$$ R_\ell(s)=\sum_{k} R_\ell(\widehat k,s) e^{ikx},\,\,\,\,\ell=1,2,3,4.$$
Above, the smoothing indexes $a_0$ and $a_1$ depend  on $\alpha$ as stated in Theorem~\ref{thm:main1} and Theorem~\ref{thm:main2}. The dependence arise only from the contribution of the resonant terms $\rho_1$ and $\rho_2$. 

Note that, with $\delta$   as in Theorem~\ref{thm:takaoka},
\begin{multline} \label{udelta2}
 \Big\|\int_0^t e^{i\alpha (t-s)\partial_x^2}\big[R_1(u)(s)+R_2(u,n)(s)\big] ds\Big\|_{L^\infty_{t\in[-\delta,\delta]} H^{s_0+a_0}_x} \\
  \lesssim \Big\|\psi_\delta(t) \int_0^t e^{i\alpha (t-s)\partial_x^2}\big[R_1(u)(s)+R_2(u,n)(s)\big] ds\Big\|_{X^{s_0+a_0,b} }
 \lesssim \|R_1(u)+R_2(u,n)\|_{X^{s_0+a_0,b-1}_\delta},
\end{multline}
for $b>1/2$.
Here we used the imbedding $X^{s_0+a_0,b}\subset L^\infty_t H^{s_0+a_0}_x$. Similarly,
\begin{multline} \label{ndelta2}
  \Big\|\int_0^t e^{-id(t-s)}\big[R_3(u,n)(s)+R_4(u,n)(s)\big] ds\Big\|_{L^\infty_{t\in[-\delta,\delta]} H^{s_1+a_1}_x}  \\
 \lesssim \|R_3(u,n)+R_4(u,n)\|_{X^{s_1+a_1,b-1}_\delta}.
\end{multline}
\begin{prop}\label{prop:schroxsb}
Given $s_1>-\frac12$, $\max(s_1,\frac{s_1}2+\frac14)\leq s_0 \leq s_1+1$, and $\frac12<b<\min(\frac34,\frac{s_0+1}2)$, we have
$$
\|R_1(u)\|_{X^{s,b-1}}\lesssim  \|u\|_{X^{s_0,\frac12}}^3,
\,\,\,\,\text{provided } s\leq s_0+\min(1,2s_0).$$
We also have
$$\|R_2(u,n)\|_{X^{s,b-1}}\lesssim \|n\|^2_{Y^{s_1,\frac12}} \|u\|_{X^{s_0,\frac12}},
$$
provided $s\leq
\min(s_0+1+2s_1,s_0+1,3+2s_1-2b,3+s_1-2b)$.
\end{prop}

\begin{prop}\label{prop:wavexsb}
Given $s_1>-\frac12$, $\max(s_1,\frac{s_1}2+\frac14)\leq s_0 \leq s_1+1$, and $\frac12<b<\frac34 +\min(0, \frac{s_0+s_1}2)$, we have
$$
\|R_3(u,n)\|_{X^{s,b-1}} + \|R_4(u,n)\|_{X^{s,b-1}} \lesssim  \|n\|_{Y^{s_1,\frac12}} \|u\|^2_{X^{s_0,\frac12}},
\,\,\,\,  $$
provided  $s\leq s_1+\min(1,2s_0,2s_0-s_1)$.
\end{prop}

We will prove these propositions later on. Using \eqref{udelta2}, \eqref{ndelta2} and the propositions  above (with $b-1/2$ sufficiently small depending on $a_0, a_1, s_0, s_1$) in
\eqref{udelta} and \eqref{ndelta}, we see that for $t\in[-\delta ,\delta ]$, we have
\begin{multline*}
\|u(t)-e^{it\alpha \partial_x^2 }u_0\|_{H^{s_0+a_0}}+\|n(t)-e^{-itd}n_0\|_{H^{s_1+a_1}}
\lesssim  \big[\|n_0\|_{H^{s_1}}+\|u_0\|_{H^{s_0}}\big]^2+\\
\big[\|n(t)\|_{H^{s_1}}+\|u(t)\|_{H^{s_0}}\big]^2+\int_0^t \big[\|n(s)\|_{H^{s_1}}+\|u(s)\|_{H^{s_0}}\big]^2 ds
+ \big[\|n\|_{Y^{s_1,\frac12}} + \|u\|_{X^{s_0,\frac12}}\big]^3.
\end{multline*}
In the rest of the proof the implicit constants depend on $\|n_0\|_{H^{s_1}}, \|u_0\|_{H^{s_0}}$.
Fix $T$ large.  For $t\leq T$, we have the bound (with $\gamma=\gamma(s_0,s_1)$)
$$\|u(t)\|_{H^{s_0}}+\|n(t)\|_{H^{s_1}}\lesssim (1+|t|)^\gamma \les T^\gamma.$$
Thus, with $\delta \sim T^{-12\gamma-}$, we have
$$
\|u(j\delta)-e^{i\delta \alpha \partial_x^2 }u((j-1)\delta)\|_{H^{s_0+a_0}}+\|n(j\delta)-e^{-i\delta d}n((j-1)\delta)\|_{H^{s_1+a_1}}  \lesssim T^{3\gamma},
$$
for any $j\leq T/\delta\sim T^{1+12\gamma+}$. Here we used the local theory bound
$$
\|u\|_{X^{s_0,1/2}_{[(j-1)\delta,\,j\delta]}} \lesssim \|u((j-1)\delta)\|_{H^{s_0}}\lesssim T^\gamma,
$$
and similarly for $n$. Using this we obtain (with $J=T/\delta\sim T^{1+12\gamma+}$)
\begin{align*}
\|u(J\delta)-e^{i\alpha J\delta\partial_x^2}u(0)\|_{H^{s_0+a_0}}&\leq \sum_{j=1}^J\|e^{i(J-j) \delta\alpha\partial_x^2}u(j\delta)-e^{i(J-j+1) \delta \alpha \partial_x^2}u((j-1)\delta)\|_{H^{s_0+a_0}}\\
&= \sum_{j=1}^J\| u(j\delta)-e^{ i\delta\alpha\partial_x^2}u((j-1)\delta)\|_{H^{s_0+a_0}}\\
&\lesssim J T^{3\gamma} \sim T^{1+15\gamma+}.
\end{align*}
The analogous bound follows similarly for the wave part $n$.

The continuity in $H^{s_0+a_0}\times H^{s_1+a_1}$ follows from dominated convergence theorem, the continuity of $u$ and $n$ in $H^{s_0}$ , $H^{s_1}$, respectively, and from the embedding $X^{s,b}\subset C^0_tH^s_x$ (for $b>1/2$). For details, see \cite{et}.
\section{Proof of Proposition~\ref{prop:schroxsb}}
First note that the denominator in the definition of $R_1$ satisfy
\begin{multline}\label{denom1}
|\alpha k^2-\alpha (k-k_1-k_2)^2-|k_1+k_2||=|\alpha| |k_1+k_2| |2k-k-k_1-\frac1\alpha \sgn(k_1+k_2)|
\\ \sim
\la k_1+k_2\ra \la 2k-k_1-k_2\ra.
\end{multline}
The last equality holds trivially in the case $1/\alpha$ is not an integer. In the case $\frac1\alpha$ is an integer it holds since the sum is over the nonresonant terms. Similarly, the denominators of $R_2$, $R_3$, $R_4$ are comparable to
\be \label{denom234}
\la k_1\ra \la 2k-k_1\ra,\,\,\,\,\,\la j\ra \la j-2j_1-2j_2\ra,\,\,\,\,\,\la j\ra \la j-2j_2\ra,
\ee
respectively.

We start with the proof for $R_2$. We have
$$
\|R_2(u,n)\|_{X^{s,b-1}}^2 =\Big\|\int_{\tau_1,\tau_2}\sum_{k_1,k_2\neq 0}^*\frac{\la k \ra^s \widehat n(k_1,\tau_1)\widehat n(k_2,\tau_2)\widehat u(k-k_1-k_2,\tau-\tau_1-\tau_2)}{(\alpha k^2-\alpha (k-k_1)^2 -|k_1|)\la \tau-k^2\ra^{1-b}}\Big\|_{\ell^2_kL^2_\tau}^2.
$$
Let
$$
f(k,\tau)=|\widehat n(k,\tau)| \la k\ra^{s_1} \la \tau-|k|\ra^{\frac12},
\,\,\,\,\,\,\,\,\,\,\,
g(k,\tau)=|\widehat u(k,\tau)| \la k\ra^{s_0} \la \tau-\alpha  k^2\ra^{\frac12}.
$$
It suffices to prove that
$$
\Big\|\int_{\tau_1,\tau_2}\sum_{k_1,k_2\neq 0}^* M(k_1,k_2,k,\tau_1,\tau_2,\tau) f(k_1,\tau_1)f(k_2,\tau_2)g(k-k_1-k_2,\tau-\tau_1-\tau_2) \Big\|_{\ell^2_kL^2_\tau}^2\lesssim \|f\|_2^4 \|g\|_2^2,
$$
where
\begin{multline}\nn
M(k_1,k_2,k,\tau_1,\tau_2,\tau)=\\
\frac{\la k \ra^s\la k_1\ra^{-s_1} \la k_2\ra^{-s_1} \la k-k_1-k_2\ra^{-s_0}  }{ (\alpha k^2-\alpha (k-k_1)^2 -|k_1|)\la \tau-\alpha  k^2\ra^{1-b}\la \tau_1-|k_1|\ra^{\frac12}\la \tau_2-|k_2|\ra^{\frac12}\la \tau-\tau_1-\tau_2-\alpha  (k-k_1-k_2)^2\ra^{\frac12}}.
\end{multline}
By Cauchy Schwarz in $\tau_1,\tau_2,k_1,k_2$ variables, we estimate the norm above by
\begin{multline*}
\sup_{k,\tau}\Big( \int_{\tau_1,\tau_2}\sum_{k_1,k_2\neq 0}^* M^2(k_1,k_2,k,\tau_1,\tau_2,\tau) \Big)\times \\ \Big\|\int_{\tau_1,\tau_2}\sum_{k_1,k_2\neq 0}  f^2(k_1,\tau_1)f^2(k_2,\tau_2)g^2(k-k_1-k_2,\tau-\tau_1-\tau_2) \Big\|_{\ell^1_kL^1_\tau}.
\end{multline*}
Note that the norm above is equal to
$\big\|  f^2*f^2*g^2 \big\|_{\ell^1_kL^1_\tau}$, which can be estimated by $ \|f\|_2^4\|g\|_2^2 $
by Young's inequality. Therefore, it suffices to prove that the supremum above is finite.

Using part b) of Lemma~\ref{lem:sums} in $\tau_1$ and $\tau_2$ integrals, we obtain
\begin{align*}
&\sup_{k,\tau} \int_{\tau_1,\tau_2}\sum_{k_1,k_2\neq 0}^* M^2 \lesssim \\
&\sup_{k,\tau}  \sum_{k_1,k_2\neq 0}^*\frac{\la k \ra^{2s} \la k_1\ra^{-2s_1} \la k_2\ra^{-2s_1} \la k-k_1-k_2\ra^{-2s_0}  }{ (\alpha k^2-\alpha (k-k_1)^2 -|k_1|)^2 \la \tau-\alpha k^2\ra^{2-2b} \la \tau-|k_1|-|k_2|-\alpha (k-k_1-k_2)^2\ra^{1-}}\\
&\lesssim \sup_{k }  \sum_{k_1,k_2\neq 0}\frac{\la k \ra^{2s} \la k_1\ra^{-2s_1} \la k_2\ra^{-2s_1} \la k-k_1-k_2\ra^{-2s_0}  }{\la k_1\ra^2 \la 2k-k_1 \ra^2  \la \alpha k^2-|k_1|-|k_2|-\alpha (k-k_1-k_2)^2\ra^{2-2b}}.
\end{align*}
The last line follows by  \eqref{denom234} and by the simple fact
\be\label{trivial}
\la \tau-n\ra \la \tau-m\ra \gtrsim \la n-m\ra.
\ee
Setting $k_2=n+k-k_1$, we rewrite the sum as
$$
\sup_{k }  \sum_{k_1\geq 0,n}\frac{\la k \ra^{2s}   \la n+k-k_1\ra^{-2s_1}   }{ \la k_1\ra^{2+2s_1}  \la 2k-k_1\ra^2 \la n\ra^{2s_0} \la \alpha (n^2-k^2) +k_1+|k_1-n-k| \ra^{2-2b}}.
$$
Here, without loss of generality (since $(k_1,k_2,k)\to (-k_1,-k_2,-k)$ is a symmetry for the sum), we only considered the case $k_1\geq 0$.\\
Case i) $-1/2<s_1<0$, $0<\frac{s_1}2+\frac14\leq s_0\leq s_1+1$.\\
We write the sum as
$$
\sum_{\stackrel{|n|\sim |k|}{k_1\geq 0}} +\sum_{\stackrel{|n|\ll |k|}{0\leq k_1\leq |n+k|}} +
\sum_{\stackrel{|n|\ll |k|}{  k_1\geq |n+k|}}+ \sum_{\stackrel{|n| \gg |k|}{ k_1\geq |n+k|}}+\sum_{\stackrel{|n| \gg |k|}{ 0\leq k_1\leq |n+k|}} =:S_1+S_2+S_3+S_4+S_5.
$$
Note that in the sum $S_1$, we have
$$
\la n\ra\sim \la k\ra,\,\,\, \la n+k-k_1\ra\les \la k_1\ra + \la 2k-k_1\ra.
$$
Using this, we have
$$
S_1\lesssim  \sum_{ k_1\geq 0 ,n} \frac{\la k \ra^{2s-2s_0} \big(\la k_1\ra^{-2s_1}+\la 2k-k_1\ra^{-2s_1}\big)  }{ \la k_1\ra^{2+2s_1}  \la 2k-k_1\ra^2    \la \alpha (n^2-k^2) +k_1+|k_1-n-k| \ra^{ 2-2b}  }. $$
Summing in $n$ using part c) of Lemma~\ref{lem:sums} and then summing in $k_1$ using
 part a) of Lemma~\ref{lem:sums}, we obtain
$$
S_1
\lesssim \la k \ra^{2s-2s_0-2-4s_1}+\la k \ra^{2s-2s_0-2-2s_1}\lesssim \la k \ra^{2s-2s_0-2-4s_1}.
$$
Note that $S_1$ is bounded in $k$ for $s\leq s_0 +1+2s_1$.

In the case of $S_2$, we have
$$|n\pm k|\sim|k|,\,\,\,\,\,\, |2k-k_1|\sim |k|,\,\,\,\,|n+k-k_1|\les |k|.$$
Also note that (since we can assume that $|k|\gg 1$)
$$
 \big|\alpha (n^2-k^2) +k_1+|k_1-n-k|\big|= \alpha (k^2-n^2)+O(|k|)\sim k^2.
$$
Using these, and then summing in $k_1$, we have
\be\nn
S_2\lesssim \sum_{\stackrel{|n|\ll |k|}{0\leq k_1\leq |n+k|}}
\frac{\la k \ra^{2s-6+4b-2s_1}     }{ \la k_1\ra^{2+2s_1}    \la n\ra^{2s_0}}   \lesssim \la k \ra^{2s-6-2s_1+4b}
\phi_{2s_0}(k)
\ee
Note that  $S_2$ is bounded in $k$ if $s< \min(s_0+\frac52+s_1-2b,3+ s_1-2b)$, and in particular,  if
$s\leq \min(s_0+1+2s_1,3+ 2s_1-2b)$.

In the case of $S_3$, we have $k_1\geq |n+k|\gtrsim |k|$.
Using this we estimate
\begin{align*}
S_3& \lesssim \sum_{\stackrel{|n|\ll |k|}{  k_1\geq |n+k|}}
\frac{\la k \ra^{2s-2-4s_1}     }{ \la 2k-k_1\ra^{2} \la n\ra^{2s_0} \la \alpha (n^2-k^2) +2k_1-n-k \ra^{2-2b}} \\ & \lesssim \sum_{ |n|\ll |k| }
\frac{\la k \ra^{2s-2-4s_1}     }{  \la n\ra^{2s_0} \la \alpha (n^2-k^2) +3k -n \ra^{2-2b}}.
\end{align*}
The second inequality follows from part a) of Lemma~\ref{lem:sums}.
Note that
$$
 \la \alpha(n^2-k^2)+3k-n\ra \sim k^2,
$$
since   $|n|\ll | k| $. Using this and then summing in $n$, we have
\begin{align*}
S_3&\lesssim \la k \ra^{2s-6-4s_1+4b}  \phi_{2s_0}(k).
\end{align*}
Note that this is also bounded in $k$  if
$s\leq \min(s_0+1+2s_1,3+ 2s_1-2b)$.

In the case of $S_4$, we have $k_1\gg |k|$. Therefore
\begin{align*}
S_4&\lesssim  \sum_{  |n|, k_1 \gg |k|}
\frac{\la k \ra^{2s-2s_0}     }{ \la k_1\ra^{4+4s_1}    \la \alpha (n^2-k^2) +2k_1-n-k \ra^{2-2b}}\\
&\lesssim  \sum_{k_1\gg |k|}\frac{\la k \ra^{2s-2s_0}     }{ \la k_1\ra^{4+4s_1} }
 \lesssim \la k \ra^{2s-2s_0-3-4s_1}.
\end{align*}
We used part c)   of Lemma~\ref{lem:sums} in the second inequality.

In the case of $S_5$, we have $|n+k-k_1|\les |n|$ and
$$
 \big|\alpha (n^2-k^2) +k_1+|k_1-n-k|\big|= \alpha (k^2-n^2)+O(|n|)\sim n^2.
$$
Thus, we estimate using part a) of Lemma~\ref{lem:sums}
\begin{align*}
S_5 \lesssim  \sum_{  |n|\gg |k|, k_1} \frac{\la k \ra^{2s}     }{ \la k_1\ra^{2+2s_1} \la 2k-k_1\ra^2   \la n\ra^{2s_0+2s_1+4-4b}}\lesssim \la k \ra^{2s-2s_0-5-4s_1+4b}.
\end{align*}
Note that to sum in $n$ we need $2s_0+2s_1+4-4b>1$, which holds under the conditions of the proposition.

Case ii) $0\leq s_1$,  $\max(s_1,\frac{s_1}2+\frac14)\leq s_0 \leq s_1+1$.\\
We write the sum as
$$
\sum_{k_1\geq 0, \,|n|\gtrsim |k|} +\sum_{|n|\ll|k|,\,0\leq k_1\ll k^2}  +
\sum_{|n|\ll|k|,\, k_1\gtrsim k^2} =:S_1+S_2+S_3.
$$
In the case of $S_1$ we have
$$
S_1\lesssim \sum_{k_1\geq 0,\, |n|\gtrsim |k|}\frac{\la k \ra^{2s-2s_0}    }{ \la k_1\ra^{2+2s_1}  \la 2k-k_1\ra^2   \la \alpha (n^2-k^2) +k_1+|k_1-n-k| \ra^{2-2b}}\lesssim \la k \ra^{2s-2s_0-2}.
$$
We obtained  the second inequality by first summing in $n$ using   part c) of Lemma~\ref{lem:sums}, and then in $k_1$ using part a) of the Lemma. Thus $S_1$ is bounded in $k$ if $s\leq s_0+1$.

In the case of $S_2$, we have
$$
\la \alpha (n^2-k^2) + k_1+|k_1-n-k| \ra\gtrsim k^2, \text{ and } \la k_1\ra \la n+k-k_1\ra \gtrsim \la n+k\ra \gtrsim \la k\ra.
$$
Therefore,
\begin{align*}
S_2 \lesssim \la k\ra^{2s-4+4b-2s_1}\sum_{|n|\ll|k|,\,0\leq k_1\ll k^2}\frac{1}{\la k_1\ra^{2 }  \la 2k-k_1\ra^2 \la n\ra^{2s_0}} \lesssim \la k\ra^{2s-6+4b-2s_1}
\phi_{2s_0}(k).
\end{align*}
Note that $S_2$ is bounded in $k$ if $s\leq \min(s_0+1,s_1+3-2b)$.

Finally we estimate $S_3$ as follows
\begin{align*}
S_3&\lesssim \sum_{|n|\ll|k|,\, k_1\gtrsim k^2} \frac{\la k \ra^{2s}       }{ \la k_1\ra^{4+4s_1}   \la \alpha (n^2-k^2) +k_1+|k_1-n-k| \ra^{2-2b}}\\& \lesssim \la k \ra^{2s-6-8s_1}     \sum_n \frac{1  }{   \la \alpha (n^2-k^2) +k_1+|k_1-n-k| \ra^{2-2b}}\lesssim \la k \ra^{2s-6-8s_1}.
\end{align*}
In the last inequality we used part c) of Lemma~\ref{lem:sums}. Note that this term is bounded in $k$ if
$s\leq s_0+1$.

We now consider $R_1$. By using Cauchy Schwarz, the convolution structure, and then integrating in $\tau_1, \tau_2$ as in the previous case, it suffices to prove that
$$
\sup_k \sum_{k_1,k_2 }^*\frac{\la k \ra^{2s} \la k_1\ra^{-2s_0} \la k_2\ra^{-2s_0} \la k-k_1-k_2\ra^{-2s_0}  |k_1+k_2|^2}{ (\alpha k^2-\alpha (k-k_1-k_2)^2-|k_1+k_2|)^2    \la   k^2-k_1^2+k_2^2- (k-k_1-k_2)^2\ra^{2-2b}}<\infty.
$$
Recalling \eqref{denom1}, and using
$$\la   k^2-k_1^2+k_2^2- (k-k_1-k_2)^2\ra\sim \la (k_1+k_2) (k-k_1)\ra,$$
it suffices to prove that
$$
\sup_k \sum_{k_1,k_2 }^* \frac{\la k \ra^{2s} \la k_1\ra^{-2s_0} \la k_2\ra^{-2s_0} \la k-k_1-k_2\ra^{-2s_0}  }{ \la 2k-k_1-k_2\ra^2   \la (k_1+k_2)(k-k_1)\ra^{2-2b}}<\infty.
$$
Note that the contribution of the case $k_1=k$ is
$$
\lesssim \sum_{k_2 } \frac{\la k \ra^{2s-2s_0}   }{ \la  k- k_2\ra^2  \la k_2\ra^{4s_0}  }\lesssim
\la k\ra^{2s-2s_0-\min(2,4s_0)},
$$
and hence it satisfies the claim. For $k_1\neq k$ (since we also have $k_1+k_2\neq 0 $ by nonresonant condition), we have $\la (k_1+k_2)(k-k_1)\ra\sim \la k_1+k_2\ra\la k-k_1\ra$.
Also letting $n=k_1+k_2$ it suffices  to consider the following sum:
\begin{multline*}
 \sum_{k_1,n} \frac{\la k \ra^{2s}   }{  \la 2k-n\ra^2  \la k-n\ra^{ 2s_0}   \la n\ra^{2-2b} \la n-k_1\ra^{ 2s_0} \la k_1\ra^{ 2s_0}  \la k-k_1\ra^{2-2b}}\\ =\sum_{|n-2k|>\frac{|k|}{2}, \, k_1}+\sum_{|n-2k|\leq \frac{|k|}{2},\, k_1}=:S_1+S_2.
\end{multline*}
We have
\begin{align*}
S_1\lesssim \la k\ra^{2s-2}\sum_{n,k_1}\frac{1 }{ \la k-n\ra^{ 2s_0}   \la n\ra^{2-2b} \la n-k_1\ra^{ 2s_0} \la k_1\ra^{ 2s_0}  \la k-k_1\ra^{2-2b}}.
\end{align*}
Using $\max(\la k-n\ra^{ 2s_0} , \la n-k_1\ra^{ 2s_0})\gtrsim \la k-k_1\ra^{2s_0}$, and then part a) of Lemma~\ref{lem:sums} (recall that $2s_0+2-2b>1$), we have
\begin{align*}
S_1&\lesssim \la k\ra^{2s-2}\sum_{n,k_1}\frac{1 }{  \la n\ra^{2-2b} \min\big(\la k-n\ra^{ 2s_0}  ,\la n-k_1\ra^{ 2s_0}\big) \la k_1\ra^{ 2s_0}  \la k-k_1\ra^{2s_0+2-2b}}\\
&\lesssim \la k\ra^{2s-2}\sum_{k_1}\frac{1 }{   \la k_1\ra^{ 2s_0}  \la k-k_1\ra^{2s_0+2-2b}} \lesssim \la k\ra^{2s-2-2s_0}.
\end{align*}
In the case of $S_2$ we have
$$
\la n\ra, \la k-n\ra \gtrsim \la k\ra,
$$
and hence
\begin{align*}
S_2&\lesssim \la k \ra^{2s-2s_0-2+2b}  \sum_{|n-2k|\leq \frac{|k|}{2},\, k_1} \frac{ 1 }{  \la 2k-n\ra^2   \la n-k_1\ra^{ 2s_0} \la k_1\ra^{ 2s_0}  \la k-k_1\ra^{2-2b}}.
\end{align*}
Note that
$$
\max(\la n-k_1\ra^{ 2s_0}, \la k_1\ra^{ 2s_0})\gtrsim \la n\ra^{2s_0}\geq \la k\ra^{2s_0}.
$$
Thus,
\begin{align*}
S_2&\lesssim \la k \ra^{2s-4s_0-2+2b}  \sum_{|n-2k|\leq \frac{|k|}{2},\, k_1} \frac{ 1 }{  \la 2k-n\ra^2   \min(\la n-k_1\ra^{ 2s_0}, \la k_1\ra^{ 2s_0})  \la k-k_1\ra^{2-2b}}.
\end{align*}
Using part a) of Lemma~\ref{lem:sums} (noting that $|n-k|\gtrsim |k|$ and that $\la k\ra^{-\gamma}\phi_\beta(k)=\la k\ra^{-\beta}\phi_\gamma(k)$ if $0<\beta,\gamma<1$), we obtain
\begin{align*}
S_2 \lesssim \la k \ra^{2s-4s_0-2+2b}  \sum_{n} \frac{ 1 }{  \la 2k-n\ra^2 } \la k\ra^{-2+2b}\phi_{2s_0}(k)
 \lesssim \la k\ra^{2s-4s_0-4+4b}  \phi_{2s_0}(k).
\end{align*}
Note that $S_2$ is bounded in $k$ if $s\leq s_0+\min(1,2s_0)$.

\section{Proof of Proposition~\ref{prop:wavexsb}}

We first consider $R_3$. By using Cauchy Schwarz, the convolution structure, and then integrating in $\tau_1, \tau_2$ as in the proof of the previous proposition, it suffices to prove that
$$
\sup_j \sum_{j_1\neq 0,j_2 }^*\frac{\la j \ra^{2s} |j|^2\la j_1\ra^{-2s_1} \la j_2\ra^{-2s_0} \la j-j_1-j_2\ra^{-2s_0}   }{ \big||j|-\alpha (j_1+j_2)^2+\alpha (j-j_1-j_2)^2\big|^2    \la
|j|-|j_1|+\alpha(j-j_1-j_2)^2-\alpha j_2^2\ra^{2-2b}}<\infty.
$$
Recalling \eqref{denom234},  it suffices to prove that
$$
\sum_{j_1\neq 0,j_2 }\frac{\la j \ra^{2s}  \la j_1\ra^{-2s_1} \la j_2\ra^{-2s_0} \la j-j_1-j_2\ra^{-2s_0}   }{ \la j-2j_1-2j_2 \ra^2    \la
|j|-|j_1|+\alpha(j-j_1-j_2)^2-\alpha j_2^2\ra^{2-2b}}
$$
is bounded in  $j$. Letting $n=j-j_1-j_2$ and $m=j_2$, we rewrite the sum as
\be\label{sumwave}
\sum_{m, n }\frac{\la j \ra^{2s}  \la j-n-m\ra^{-2s_1}    }{ \la 2n-j \ra^2  \la m\ra^{ 2s_0} \la n\ra^{ 2s_0}  \la
\alpha n^2-\alpha m^2+|j|-|j-n-m|+\ra^{2-2b}}.
\ee
We note that a similar argument gives us the following sum for $R_4$:
\be\label{sumwave2}
\sum_{m, n }\frac{\la j \ra^{2s}  \la j-n-m\ra^{-2s_1}    }{ \la 2n-j \ra^2  \la m\ra^{ 2s_0} \la n\ra^{ 2s_0}  \la
\alpha n^2-\alpha m^2-|j|-|j-n-m|+\ra^{2-2b}}.
\ee
We note that, by symmetry, if we can prove that
\be\label{sumwave3}
\sum_{m, n }\frac{\la j \ra^{2s}  \la j-n-m\ra^{-2s_1}    }{ \la 2n-j \ra^2  \la m\ra^{ 2s_0} \la n\ra^{ 2s_0}  \la
\alpha n^2-\alpha m^2+j-|j-n-m|+\ra^{2-2b}}
\ee
is bounded in $j\neq 0$, then the boundedness of \eqref{sumwave} and \eqref{sumwave2} follow.

Case i) $-\frac12 <s_1<0$.\\
We rewrite \eqref{sumwave}  as
$$
\sum_{|n|\sim |m|\les |j|}  + \sum_{|n|\sim |m|\gg |j|}  + \sum_{\stackrel{|n|\ll |m|}{|j|\geq |m+n|}} +
\sum_{\stackrel{|n|\gg |m|}{|j|\geq |m+n|}}+\sum_{\stackrel{|n|\ll |m|}{|j|\leq |m+n|}}+\sum_{\stackrel{|n|\gg |m|}{|j|\leq |m+n|}}  =:S_1+S_2+S_3+S_4+S_5+S_6.
$$

For $S_1$ we have
$$
S_1\les \sum_{|n|\sim |m|\les |j|} \frac{\la j \ra^{2s-2s_1}      }{ \la 2n-j \ra^2  \la n\ra^{ 4s_0}  \la
j-|j-n-m|+\alpha n^2-\alpha m^2\ra^{2-2b}}\les \la j \ra^{2s-2s_1-\min(2,4s_0)}.
$$
In the second inequality we first summed in $m$ using part c) of Lemma~\ref{lem:sums}, and then in $n$ using part a) of the lemma.

For $S_2$ we have
$$
S_2\les \sum_{|n|\sim |m|\gg |j|} \frac{\la j \ra^{2s }      }{ \la n \ra^{2+4s_0+2s_1}  \la
j-|j-n-m|+\alpha n^2-\alpha m^2\ra^{2-2b}}\les \la j \ra^{2s-2s_1- 4s_0-1}.
$$
Again, we first summed in $m$ using part c) of Lemma~\ref{lem:sums}.

In the case of $S_3$ we have $|n|\ll|m|\les |j| $, and hence
\begin{multline*}
S_3\les \sum_{|n|\ll |m| \les|j| } \frac{\la j \ra^{2s -2s_1-2}      }{   \la n\ra^{ 4s_0}\la
j-|j-n-m|+\alpha n^2-\alpha m^2\ra^{2-2b} } \\ \les \sum_{|n| \les |j|} \frac{\la j \ra^{2s -2s_1-2}      }{   \la n\ra^{ 4s_0 } }\les \la j \ra^{2s -2s_1- 2 } \phi_{4s_0}(j)\les \la j \ra^{2s-2s_1-\min(2,4s_0)}.
\end{multline*}

In the case of $S_4$ we have
$$
\la 2n-j\ra +\la j-|j-n-m|+\alpha n^2-\alpha m^2\ra \gtrsim n^2.
$$
Since $\la 2n-j\ra\gtrsim n^2$ implies that $\la 2n-j\ra\gtrsim \la j\ra$, we have
\begin{multline*}
\frac1{\la2n-j\ra^2 \la j-|j-n-m|+\alpha n^2-\alpha m^2\ra^{2-2b}} \\ \les\frac1{\la j\ra^2 \la j-|j-n-m|+\alpha n^2-\alpha m^2\ra^{2-2b}}+\frac1{\la2n-j\ra^2\la n\ra^{4-4b}}.
\end{multline*}
Therefore we estimate
\begin{multline*}
S_4\les \sum_{|m|\ll |n| \les |j|} \frac{\la j \ra^{2s -2s_1-2}      }{   \la m\ra^{ 4s_0 } \la j-|j-n-m|+\alpha n^2-\alpha m^2\ra^{2-2b} } \\ +  \sum_{|m|\ll |n| \les |j|} \frac{\la j \ra^{2s -2s_1}      }{ \la 2n-j \ra^2   \la n\ra^{ 2s_0+4-4b} \la m\ra^{2s_0} }.
\end{multline*}
The first line above can be estimated as in $S_3$ switching the roles of $n$ and $m$. To estimate the second line first sum in $n$ using part a) of Lemma~\ref{lem:sums}, and then in $m$ to obtain
$$
\les \la j \ra^{2s -2s_1-\min(2,2s_0+4-4b)} \phi_{2s_0}(j) \les  \la j \ra^{2s -2s_1-\min(2,4s_0)}.
$$
In the case of $S_5$, we have
\be \nn
\la j-|j-n-m|+\alpha n^2-\alpha m^2\ra \sim \la m\ra^2,\,\,\,\,\,\,\,|m|\gtrsim |j|.
\ee

Therefore, noting that $2s_0+2s_1+4-4b>1$, we have
\begin{multline*}
S_5\les \sum_{|n|\ll |m| } \frac{\la j \ra^{2s  }      }{ \la 2n-j \ra^2  \la n\ra^{ 2s_0} \la m\ra^{ 2s_0+2s_1+4-4b} } \\ \les \sum_{n } \frac{\la j \ra^{2s}      }{ \la 2n-j \ra^2   \la n\ra^{ 4s_0+2s_1+3-4b} }\les \la j \ra^{2s -\min(2,4s_0+2s_1+3-4b)}.
\end{multline*}

In the case of $S_6$, we have
\be \label{brackn}
\la j-|j-n-m|+\alpha n^2-\alpha m^2\ra \sim \la n\ra^2,\,\,\,\,\,\,\,|n|\gtrsim |j|.
\ee
Therefore,
\begin{multline*}
S_6\les \sum_{|m|\ll |n| \gtrsim|j|} \frac{\la j \ra^{2s }      }{ \la 2n-j \ra^2  \la n\ra^{2s_0+2s_1+4-4b} \la m\ra^{ 2s_0}}  \\
 \les \sum_{|n|\gtrsim |j| } \frac{\la j \ra^{2s }   \phi_{2s_0}(n)   }{ \la 2n-j \ra^2
\la n\ra^{2s_0+2s_1+4-4b} }   \les \la j \ra^{2s -2s_0-2s_1-4+4b} \phi_{2s_0}(j).
\end{multline*}
In the last inequality we used $|n|\gtrsim |j|$ and then summed in $n$.

Case ii) $s_1\geq 0$.\\
We rewrite \eqref{sumwave}  as
$$
\sum_{|n|\les |m|}  + \sum_{|m| \ll |n| \ll |j| }  + \sum_{ |m|\ll |n| \gtrsim |j|}  =:S_1+S_2+S_3.
$$
In the case of $S_1$, we have $ |j|\leq |j-n-m|+|m+n|\les |j-n-m|+|m|$, and hence
$$
\la j-n-m\ra \la m\ra \gtrsim \la j\ra.
$$
Using this and noting that $s_0\geq s_1$, we have
\begin{multline*}
S_1\les \sum_{|n|\les |m| } \frac{\la j \ra^{2s-2s_1  }      }{ \la 2n-j \ra^2  \la n\ra^{ 4s_0-2s_1}   \la
j-|j-n-m|+\alpha n^2-\alpha m^2\ra^{2-2b}} \\ \les    \la j \ra^{2s-2s_1-\min(2,4s_0-2s_1)}.
\end{multline*}
In the last inequality we summed in $m$ using part c) of Lemma~\ref{lem:sums} and then in $n$ using part a) of the lemma.

In the case of $S_2$
we have
$$
S_2\les \sum_{|m|\ll |n| \ll |j| } \frac{\la j \ra^{2s -2-2s_1  }      }{     \la m\ra^{4s_0} \la
j-|j-n-m|+\alpha n^2-\alpha m^2\ra^{2-2b}}\les \la j \ra^{2s -2-2s_1  } \phi_{4s_0}(j).
$$

Note that in the case of $S_3$ we have \eqref{brackn}. Therefore
$$
S_3\les \sum_{|m|\ll |n| \gtrsim |j|} \frac{\la j \ra^{2s   }      }{ \la 2n-j \ra^2  \la n\ra^{ 2s_0+4-4b} \la m\ra^{2s_0}  \la j-n-m\ra^{2s_1}  }.
$$
If $s_0+s_1>1/2$, we sum in $m$ and then in $n$ using part a) of Lemma~\ref{lem:sums} to obtain
$$
S_3\les \sum_{ |n| \gtrsim |j|} \frac{\la j \ra^{2s-2s_0-4+4b   }     }{ \la 2n-j \ra^2    \la j-n \ra^{2s_1+\min(0,2s_0-1)-}  }\les \la j \ra^{2s-2s_0-4+4b -\min(2,2s_1,2s_1+2s_0-1)+  }.
$$
If $s_0+s_1\in(0,1/2]$, we have
$$
S_3 \les \sum_{ |n| \gtrsim |j|} \frac{\la j \ra^{2s }    \la n\ra^{1-2s_0-2s_1+}   }{ \la 2n-j \ra^2   \la n\ra^{ 2s_0+4-4b}    }\les \la j \ra^{2s-4s_0-2s_1-3+4b+   }.
$$
To estimate the second line

Note that each term above is bounded in $j$ if $s\leq s_1+\min(1,2s_0-s_1)$.

\section{Existence of Global Attractor}
In this section we prove   Theorem~\ref{thm:attractor}. As in the    previous sections  we drop the `$\pm$' signs and
work with the system:
\begin{equation}\label{eq:fdzakharov1n}
\left\{
\begin{array}{l}
 (i\partial_t+\alpha\partial_x^2+i\gamma)u =n  u+ f , \,\,\,\,  x \in {\mathbb T}, \,\,\,\,  t\in [-T,T],\\
 (i\partial_t -  d+i\gamma)n  =   d (|u|^2), \\
 u(x,0)=u_0(x) \in H^{1}(\mathbb T), \,\,\,\, 
n(x,0)=n_{0}(x) \in \dot L^2(\mathbb T).
\end{array}
\right.
\end{equation}

We start with a smoothing  estimate for \eqref{eq:fdzakharov1n} which implies the existence of a global attractor:
\begin{theorem}\label{thm:fdsmooth}
Consider the solution of \eqref{eq:fdzakharov1n} with initial data $(u_0,n_{ 0})\in H^{1}\times \dot L^2$.
Then, for $\frac1\alpha\not\in \N$, and for any $a <1 $, we have
\begin{align}\label{fdusmooth}
u(t)-e^{i\alpha t\partial_x^2-\gamma t}u_0 &\in C^0_tH_x^{1+ a }([0,\infty)\times \T  ),\\ n(t)-e^{- itd-\gamma t}n_{ 0}  &\in C^0_tH_x^{a}([0,\infty)\times \T  ).\label{fdnsmooth}
\end{align}
Moreover,
\be \label{fdgrowth1}
\|u(t)-e^{i\alpha t\partial_x^2-\gamma t}u_0\|_{ H^{1+a}} +\| n (t)-e^{-itd-\gamma t}n_{ 0}\|_{H^{ a}}  \\ \leq  C\big( a,\alpha, \gamma,\|f\|_{H_1},\| u_0\|_{H^{1}},\|n_{ 0}\|_{L^2}  \big).
\ee
In the case $ \alpha=1$ we have, for any $a <1 $,
\begin{multline}\label{fdgrowth2}
\Big\|u(t)-e^{i  t\partial_x^2-\gamma t}u_0+i\int_0^t e^{(i \partial_x^2  -\gamma) (t-t^\prime)}\rho_1 dt^\prime\Big\|_{ H^{1+a}}   +\big\| n(t)-e^{-itd-\gamma t}n_{ 0} \big\|_{H^{ a}}  \\ \leq C\big( a, \gamma,\|f\|_{H_1},\| u_0\|_{H^{1}},\|n_{ 0}\|_{L^2}  \big),
\end{multline}
where $\rho_1$ is as in Proposition~\ref{thm:dbp}. The analogous continuity statements as in \eqref{fdusmooth}, \eqref{fdnsmooth} are also valid.
\end{theorem}
\begin{proof}
Writing
$$u(x,t)=\sum_{k } u_k(t) e^{ikx},\,\,\,\,\,n(x,t)=\sum_{j\neq 0} n_j(t) e^{ijx},\,\,\,\,\,f(x )=\sum_{k } f_k(t) e^{ikx}$$
we obtain the following system for the Fourier coefficients:
\begin{equation}\label{eq:fdzakharov3}
\left\{
\begin{array}{l}
 i\partial_t  u_k +(i\gamma- \alpha k^2) u_k =\sum_{k_1+k_2=k,\,k_1\neq 0} n_{k_1} u_{k_2} +  f_k, \\
 i\partial_t n_j +(i\gamma - |j|) n_j = |j| \sum_{j_1+j_2=j} u_{j_1}  \overline{u_{-j_2} }.
\end{array}
\right.
\end{equation}

We have the following  proposition which follows from differentiation by parts as in Proposition~\ref{thm:dbp}
by using the change of variables $m_j=n_je^{i|j|t+\gamma t}$, and $v_k=u_ke^{i\alpha k^2 t+\gamma t}$.
\begin{prop}\label{thm:dbpfd}
The system \eqref{eq:fdzakharov3} can be written in the following form:
\begin{multline}\label{fdv_eq_dbp}
i\partial_t\big[e^{it\alpha k^2+\gamma t}u_k\big]+ie^{-\gamma t}\partial_t\big[e^{it\alpha k^2+2\gamma t}B_1(n,u)_k\big] =\\ e^{it\alpha k^2+\gamma t}\big[\rho_1(k)+ f_k+ B_1(n,f)+ R_1(u)(\widehat k, t)+R_2(u,n)(\widehat k, t)\big],
\end{multline}
\begin{multline}
 i\partial_t \big[e^{it|j|+\gamma t}n_j\big] +ie^{-\gamma t} \partial_t \big[e^{it|j|+2\gamma t} B_2(u)_j\big]  =  \\  e^{it|j|+\gamma t}\big[ \rho_2(j)  +B_2(f,u)
+ B_2(u, f )+R_3(u,n)(\widehat j, t)+R_4(u,n)(\widehat j, t) \big]. \label{fdm_eq_dbp}
\end{multline}
where $B_i, \rho_i$, $i=1,2$, and $R_j$, $j=1,2,3,4$ are as in Proposition~\ref{thm:dbp}.
\end{prop}

Integrating \eqref{fdv_eq_dbp} from $0$ to $t$, we obtain
\begin{multline*}
u_k(t)-e^{-it\alpha k^2-\gamma t}u_k(0)=  -B_1(n,u)_k +
e^{-it\alpha k^2-\gamma t}B_1(n_0,u_0)_k +\\ \int_0^t e^{-(i\alpha k^2 +\gamma)(t-t^\prime)} \Big[ -\gamma B_1(n,u)_k-i\rho_1(k) -i f_k-i B_1(n,f)_k   \Big] dt^\prime \\
-i\int_0^t e^{-(i\alpha k^2 +\gamma)(t-t^\prime)} \big[ R_1(u)(\widehat k, t^\prime)+R_2(u,n)(\widehat k, t^\prime)\big] dt^\prime.
\end{multline*}
First note that
\be\label{fcont}
\Big\|\int_0^t e^{-(i\alpha k^2 +\gamma)(t-t^\prime)} f_k dt^\prime\Big\|_{H^{1+a}}=\Big\|\frac{f_k}{i\alpha k^2+\gamma} (1-e^{-it\alpha k^2-\gamma t}) \Big\|_{H^{1+a}} \les \|f\|_{H^{a-1}}.
\ee
In the case $\frac1\alpha\not \in \N$,  using \eqref{fcont}, the estimates in Lemma~\ref{apriori} and Proposition~\ref{prop:schroxsb} as above, and also using the growth bound in \eqref{aprioribd},  we obtain for any $a<1$
\begin{multline*}
\|u(t)-e^{i\alpha\partial_x^2 t-\gamma t} u_0\|_{H^{1+a}}\les \|f\|_{H^{a-1}}+ \big[\|f\|_{H^1}+\|n(0)\|_{L^2}+\|u(0)\|_{H^1}\big]^2 + \big[\|u\|_{X^{1,\frac12}_\delta}+\|n\|_{Y_\delta^{1,\frac12}}\big]^3.
\end{multline*}
Using the local theory bound for $X^{1,\frac12}_\delta, Y^{1,\frac12}_\delta$ norms for a
$\delta =\delta( \|n_0\|_{L^2}, \|u_0\|_{H^1},\|f\|_{H^1})$, we obtain for $t<\delta$
$$
\|u(t)-e^{i\alpha\partial_x^2 t-\gamma t} u_0\|_{H^{1+a}}\les C(a, \gamma, \|f\|_{H^1}, \|n_0\|_{L^2}+\|u_0\|_{H^1}).
$$
In the rest of the proof the implicit constants depend on $a, \gamma, \|f\|_{H^1}, \|n_0\|_{L^2}+\|u_0\|_{H^1}$.
Fix $t$ large, and $\delta$ as above.  We have
$$
\|u(j\delta)-e^{i\alpha\partial_x^2 \delta-\gamma \delta}u((j-1)\delta)\|_{H^{1+a}}\lesssim 1,
$$
for any $j\leq t/\delta$.
Using this we obtain (with $J=t/\delta$)
\begin{align*}
&\|u(J\delta)-e^{J\delta(i\alpha\partial_x^2 -\gamma)}u(0)\|_{H^{1+a}} \leq \sum_{j=1}^J\|e^{(J-j)\delta(i\alpha\partial_x^2 -\gamma ) }u(j\delta)-e^{(J-j+1)\delta(i\alpha\partial_x^2 -\gamma )}u((j-1)\delta)\|_{H^{1+a}}\\
&= \sum_{j=1}^Je^{-(J-j) \delta \gamma}\| u(j\delta)-e^{ \delta (i\alpha\partial_x^2 -\gamma)}u((j-1)\delta)\|_{H^{1+a}}\lesssim  \sum_{j=1}^Je^{-(J-j) \delta \gamma} \lesssim \frac{1}{1-e^{-\delta\gamma}}.
\end{align*}

In the case $ \alpha=1$, we have to separate the resonant term in this argument.
We have the following inequality for $t<\delta$
$$
\Big\|u(t)-e^{i\alpha\partial_x^2 t-\gamma t} u_0+i\int_0^t e^{(i\alpha\partial_x^2  -\gamma) (t-t^\prime)}\rho_1 dt^\prime\Big\|_{H^{1+a}}\les C(a, \gamma, \|f\|_{H^1}, \|n_0\|_{L^2}+\|u_0\|_{H^1}).
$$
Accordingly we have
\begin{multline*}
 \Big\|u(J\delta)-e^{J\delta(i\alpha\partial_x^2 -\gamma)}u(0)+\int_0^{J\delta}e^{(i\alpha\partial_x^2  -\gamma) (J\delta-t^\prime)}\rho_1 dt^\prime\Big\|_{H^{1+a}} \leq \\ \sum_{j=1}^J\Big\|e^{(J-j)\delta(i\alpha\partial_x^2 -\gamma ) }\Big( u(j\delta)-e^{ \delta(i\alpha\partial_x^2 -\gamma )}u((j-1)\delta)+i  \int_{(j-1)\delta}^{j\delta}e^{(i\alpha\partial_x^2  -\gamma) (j\delta-t^\prime)}\rho_1 dt^\prime\Big)\Big\|_{H^{1+a}}= \\
 \sum_{j=1}^Je^{-(J-j) \delta \gamma}\Big\| u(j\delta)-e^{ \delta (i\alpha\partial_x^2 -\gamma)}u((j-1)\delta)+i  \int_{(j-1)\delta}^{j\delta}e^{(i\alpha\partial_x^2  -\gamma) (j\delta-t^\prime)}\rho_1 dt^\prime\Big\|_{H^{1+a}}\lesssim   \\
\sum_{j=1}^Je^{-(J-j) \delta \gamma} \lesssim \frac{1}{1-e^{-\delta\gamma}}.
\end{multline*}
The corresponding inequalities for the wave part follow similarly. The only difference is that we don't need to separate the resonant term, since $\rho_2\in H^1$ by Lemma~\ref{apriori}.

This completes the proof of the global bound stated in Theorem~\ref{thm:fdsmooth}. Finally the continuity in  in $H^{1}\times \dot L^2$ follows as in \cite{et}. We omit the details.
\end{proof}
\begin{proof}[Proof of Theorem~\ref{thm:attractor}]
We start with the case $\frac1\alpha \not \in \N$. 
First of all note that the existence of an absorbing set, $\mathcal B_0\subset H^1\times \dot L^2$, is immediate from \eqref{aprioribd}. Second, we need to verify the asymptotic compactness of the propagator $U_t$. It suffices to prove that for any sequence $t_r\to\infty$ and for any sequence $(u_{0,r},n_{0,r})$ in $ \mathcal B_0$, the sequence
$U_{t_r}(u_{0,r},n_{0,r})$ has a convergent subsequence in $H^1\times \dot L^2$.

To see this note that by Theorem~\ref{thm:fdsmooth}, (if $(u_{0},n_{0})\in \mathcal B_0$)
$$U_t\big(u_{0},n_{0}\big)=\big(e^{i\alpha t\partial_x^2-\gamma t}u_{0}, e^{- itd-\gamma t} n_{0} \big)+ N_t\big(u_{0},n_{0}\big)$$
where $N_t\big(u_{0},n_{0}\big)$ is in a ball in $H^{1+a}\times H^a$ with radius depending on $a\in(0,1), \alpha, \gamma,$ and $ \|f\|_{H^1}$. By Rellich's theorem, $\{N_t\big(u_{0},n_{0}\big):t>0, (u_{0},n_{0})\in \mathcal B_0\}$ is precompact in $H^1\times  \dot L^2$. Since
$$ \big\|\big(e^{i\alpha t \partial_x^2-\gamma t }u_{0}, e^{- itd-\gamma t} n_{0} \big)\big\|_{H^1\times \dot L^2}\les e^{-\gamma t}\to 0,\,\,\,\,\text{ as } t\to \infty,$$
uniformly on $\mathcal B_0$, we conclude that $\{U_{t_r}\big(u_{0,r},n_{0,r}\big):r\in \N\}$ is precompact in $H^1\times  \dot L^2$. Thus, $U_t$ is asymptotically compact.  This and Theorem A imply the existence of a global attractor $\mathcal A\subset H^1\times \dot L^2$.

We now prove that the attractor set $\mathcal A$ is a compact subset of $H^{1+a}\times H^a$ for any $a\in(0,1)$. By Rellich's theorem, it suffices to prove that for any $a\in(0,1)$, there exists a closed ball $B_a\subset  H^{1+a}\times H^a$  of radius $C(a,\alpha, \gamma,\|f\|_{H^1})$ such that  $\mathcal A\subset B_a$.
By definition
$$
\mathcal A=\bigcap_{\tau\geq 0}\overline{\bigcup_{t\geq \tau} U_t\mathcal B_0}=: \bigcap_{\tau\geq 0} V_\tau.
$$
By Theorem~\ref{thm:fdsmooth} and the discussion above, $V_\tau$ is contained in a $\delta_\tau$ neighborhood, $N_\tau$, of a ball $B_a$ in $H^1\times \dot L^2$ whose radius depends only on $a,\alpha,\gamma,\|f\|_{H^1}$, and where   $\delta_\tau\to 0$ as $\tau$ tends to infinity. Since   $B_a$ is a compact subset of
$H^1\times \dot L^2$, we have
$$
\mathcal A = \bigcap_{\tau\geq 0} V_\tau\subset \bigcap_{\tau>0} N_\tau =B_a.
$$

Now consider the case $\frac1\alpha\in \N$. For simplicity, we take $\alpha=1$. 
We have to be slightly more careful in this case because of the contribution of the resonant term, $\rho_1$, which is does not belong to $H^{1+a}$ for any $a>0$.  Recall that,   by Theorem~\ref{thm:fdsmooth},  for $(u_0,n_0)\in\mathcal B_0$
\be \label{reson_evol}
U_t\big(u_{0},n_{0}\big)=\big(e^{i\alpha t\partial_x^2-\gamma t}u_{0}, e^{- itd-\gamma t} n_{0} \big)+ N_t\big(u_{0},n_{0}\big)
+  i\Big(\int_0^t e^{(i \partial_x^2  -\gamma) (t-t^\prime)}\rho_1 dt^\prime, 0\Big),
\ee
where $N_t\big(u_{0},n_{0}\big)$ is in a ball in $H^{1+a}\times H^a$ with radius depending on $a\in(0,1), \gamma,$ and $ \|f\|_{H^1}$.
Recall from Proposition~\ref{thm:dbp}, that the Fourier coefficients of $\rho_1$   are
$$
(\rho_1)_k  = \rho_1(n ,u )_k=n_{2k-\sgn(k)} u_{ \sgn(k)-k},\,\,k\neq 0.
$$
In light of the proof of the case $\frac1\alpha\not\in\N$ above, it suffices  to consider the contribution of the resonant term under the assumption that $(u_0,n_0)\in\mathcal B_0$.
Using \eqref{reson_evol}, we write
\be\label{rho1bd}
\rho_1\big(n(t^\prime),u(t^\prime)\big)=\rho_1\big(e^{- it^\prime d-\gamma t^\prime} n_{0} , u(t^\prime)\big) + \rho_1\big(N_{t^\prime} (n_{0}), u(t^\prime)\big).
\ee
Now note that, by Lemma~\ref{apriori}, we have
$$
\big\|\rho_1(n ,u ) \big\|_{H^{1+a}}\les \|n\|_{H^a}  \|u\|_{H^1}. 
$$ 
Using this with $a=0$, we see that the contribution of the first summand in \eqref{rho1bd} to the resonant term in \eqref{reson_evol} satisfies
\begin{multline*}
\Big\|\int_0^t e^{(i \partial_x^2  -\gamma) (t-t^\prime)}\rho_1\big(e^{- it^\prime d-\gamma t^\prime} n_{0} , u(t^\prime)\big) dt^\prime\Big\|_{H^1} \les \int_0^t e^{ -\gamma  (t-t^\prime)}\|e^{- it^\prime d-\gamma t^\prime} n_{0} \|_{L^2} \|u(t^\prime)\|_{H^1} dt^\prime \\ \leq  t e^{-\gamma t} C(a,\gamma,\|f\|_{H^1}),
\end{multline*}
which goes to zero uniformly in $\mathcal B_0$. Similarly, the contribution of the second summand in \eqref{rho1bd} to the resonant term in \eqref{reson_evol} satisfies
\begin{multline*}
\Big\|\int_0^t e^{(i \partial_x^2  -\gamma) (t-t^\prime)}\rho_1\big(N_{t^\prime} (n_{0}), u(t^\prime)\big) dt^\prime\Big\|_{H^{1+a}} \les \int_0^t e^{ -\gamma  (t-t^\prime)}\|N_{t^\prime} (n_{0}) \|_{H^a} \|u(t^\prime)\|_{H^1} dt^\prime \\ \leq C(a,\gamma,\|f\|_{H^1}).
\end{multline*}
The rest of the proof is same as the case $\frac1\alpha\not\in\N$.
\end{proof}

\section{Appendix}
We prove Lemma~\ref{lem:sums}. Note that, with $m=k_2-k_1$, we can rewrite the sum in part a) as
$$
\sum_n\frac1{\la n\ra^\beta \la n-m\ra^\gamma}.
$$
For $|n|<|m|/2$, we estimate the sum by
$$
\sum_{|n|<|m|/2} \frac1{\la n\ra^\beta \la m\ra^\gamma} \leq \la m\ra^{-\gamma}\phi_\beta(m).
$$
For $|n|>2|m|$, we estimate by
$$
\sum_{|n|>2|m|} \frac1{\la n\ra^{\beta+\gamma}} \les \la m\ra^{1-\beta-\gamma}\les \la m\ra^{-\gamma} \phi_\beta(m).
$$
Finally for  $|n|\sim|m|$, we estimate by
$$
\sum_{|n|\sim |m|} \frac{1}{\la m\ra^\beta \la n-m\ra^{\gamma}}\les \la m\ra^{-\beta}\phi_\gamma(m)\les \la m\ra^{-\gamma}\phi_\beta(m).
$$
The last inequality follows from the definition of $\phi_\beta$ and the hypothesis $\beta\geq \gamma$.

The part b) follows from part a). To obtain part c), write
$$
|n^2+c_1n+c_2|=|(n+z_1)(n+z_2)|\geq |n+x_1| |n+x_2|
$$
where $x_i$ is the real part of $z_i$. The contribution of the terms $|n+x_1|<1$ or $|n+x_2|<1$
is $\les 1$. Therefore, we estimate the sum in part c) by
$$
\les 1+ \sum_n\frac{1}{\la n+x_1\ra^\beta \la n+x_2\ra^\beta}\les 1
$$
by part a).

 \end{document}